\documentclass[11pt]{article}%
\usepackage{graphicx}
\usepackage{amsmath}
\usepackage{amsfonts}
\usepackage{amssymb}%
\providecommand{\U}[1]{\protect\rule{.1in}{.1in}}
\newtheorem{theorem}{Theorem}[section]

\newtheorem{conjecture}[theorem]{Conjecture}
\newtheorem{corollary}[theorem]{Corollary}

\newtheorem{lemma}[theorem]{Lemma}

\newtheorem{proposition}[theorem]{Proposition}
\newtheorem{remark}[theorem]{Remark}

\newenvironment{proof}[1][Proof]{\textbf{#1.} }{\hfill\rule{0.5em}{0.5em}}
{\catcode`\@=11\global\let\AddToReset=\@addtoreset
\AddToReset{equation}{section}

\AddToReset{theorem}{section}

\begin{document}

\title{A new dynamical approach of Emden-Fowler equations and systems }
\author{Marie Fran\c{c}oise BIDAUT-VERON\thanks{Laboratoire de Math\'{e}matiques et
Physique Th\'{e}orique, CNRS UMR 6083, Facult\'{e} des Sciences, 37200 Tours
France. E-mail address:veronmf@univ-tours.fr}
\and Hector GIACOMINI\thanks{Laboratoire de Math\'{e}matiques et Physique
Th\'{e}orique, CNRS UMR 6083, Facult\'{e} des Sciences, 37200 Tours France.
E-mail address:Hector.Giacomini@lmpt.univ-tours.fr}}
\date{.}
\maketitle

\begin{abstract}
We give a new approach on general systems of the form
\[
(G)\left\{
\begin{array}
[c]{c}%
-\Delta_{p}u=-\operatorname{div}(\left\vert \nabla u\right\vert ^{p-2}\nabla
u)=\varepsilon_{1}\left\vert x\right\vert ^{a}u^{s}v^{\delta},\\
-\Delta_{q}v=-\operatorname{div}(\left\vert \nabla v\right\vert ^{q-2}\nabla
u)=\varepsilon_{2}\left\vert x\right\vert ^{b}u^{\mu}v^{m},
\end{array}
\right.
\]
where $Q,p,q,\delta,\mu,s,m,$ $a,b$ are real parameters, $Q,p,q\neq1,$ and
$\varepsilon_{1}=\pm1,$ $\varepsilon_{2}=\pm1.$ In the radial case we reduce
the problem to a quadratic system of four coupled first order autonomous
equations, of Kolmogorov type. It allows to obtain new local and global
existence or nonexistence results. We consider in particular the case
$\varepsilon_{1}=\varepsilon_{2}=1.$ We describe the behaviour of the ground
states in two cases where the system is variational. We give a result of
existence of ground states for a nonvariational system with $p=q=2$ and
$s=m>0,$ that improves the former ones. It is obtained by introducing a new
type of energy function. In the nonradial case we solve a conjecture of
nonexistence of ground states for the system with $p=q=2$, $\delta=m+1$ and
$\mu=s+1.\bigskip$

\textbf{Keywords }Elliptic quasilinear systems. Variational or nonvariational
problems. Autonomous and quadratic systems. Stable manifolds. Heteroclinic
orbits.\bigskip

\textbf{A.M.S. Subject Classification }34B15, 34C20, 34C37; 35J20, 35J55,
35J65, 35J70; 37J45.

\end{abstract}

.\pagebreak\medskip

\section{Introduction}

In this paper we consider the nonnegative solutions of Emden-Fowler equations
or systems in $\mathbb{R}^{N}(N\geqq1)$,%
\begin{equation}
-\Delta_{p}u=-\operatorname{div}(\left\vert \nabla u\right\vert ^{p-2}\nabla
u)=\varepsilon_{1}\left\vert x\right\vert ^{a}u^{Q}, \label{one}%
\end{equation}%
\begin{equation}
(G)\left\{
\begin{array}
[c]{c}%
-\Delta_{p}u=-\operatorname{div}(\left\vert \nabla u\right\vert ^{p-2}\nabla
u)=\varepsilon_{1}\left\vert x\right\vert ^{a}u^{s}v^{\delta},\\
-\Delta_{q}v=-\operatorname{div}(\left\vert \nabla v\right\vert ^{q-2}\nabla
u)=\varepsilon_{2}\left\vert x\right\vert ^{b}u^{\mu}v^{m},
\end{array}
\right.  \label{gen}%
\end{equation}
where $Q,p,q,\delta,\mu,s,m,$ $a,b$ are real parameters, $Q,p,q\neq1,$ and
$\varepsilon_{1}=\pm1,$ $\varepsilon_{2}=\pm1.$ These problems are the subject
of a very rich litterature, either in the case of source terms $(\varepsilon
_{1}=\varepsilon_{2}=1)$ or absorption terms $(\varepsilon_{1}=\varepsilon
_{2}=1)$ or mixed terms $(\varepsilon_{1}=-\varepsilon_{2}).$ In the sequel we
are concerned by the radial solutions, except at Section 9 where the solutions
may be nonradial.\medskip

In this article we we give a \underline{\textit{new way of studying the radial
solutions}}. In Section \ref{sys} we reduce system $(G)$ to a quadratic
autonomous system:%
\[
(M)\left\{
\begin{array}
[c]{c}%
X_{t}=X\left[  X-\frac{N-p}{p-1}+\frac{Z}{p-1}\right]  ,\\
Y_{t}=Y\left[  Y-\frac{N-q}{q-1}+\frac{W}{q-1}\right]  ,\\
Z_{t}=Z\left[  N+a-sX-\delta Y-Z\right]  ,\\
W_{t}=W\left[  N+b-\mu X-mY-W\right]  ,
\end{array}
\right.
\]
where $t=\ln r,$ and
\begin{equation}
X(t)=-\frac{ru^{\prime}}{u},\quad Y(t)=-\frac{rv^{\prime}}{v},\quad
Z(t)=-\varepsilon_{1}r^{1+a}u^{s}v^{\delta}\frac{u^{\prime}}{\left\vert
u^{\prime}\right\vert ^{p}},\quad W(t)=-\varepsilon_{2}r^{1+b}u^{\mu}%
v^{m}\frac{v^{\prime}}{\left\vert v^{\prime}\right\vert ^{q}}. \label{sub}%
\end{equation}
This system is of Kolmogorov type. The reduction is valid for equations and
systems with source terms , absorption terms , or mixed terms $.$ It is
remarkable that in the new system, $p$ and $q$ appear only as \underline
{\textit{simple coefficients}}, which allows to treat any value of the
parameters, even $p$ or $q<1,$ and $s,m,\delta$ or $\mu<0$.$\medskip$

In Section \ref{scal} we revisit the well-known scalar case (\ref{one}), where
$(G)$ becomes two-dimensional. We show that the phase plane of the system
gives \underline{\textit{at the same time}} the behaviour of the two equations%
\[
-\Delta_{p}u=\left\vert x\right\vert ^{a}u^{Q}\text{ and }-\Delta
_{p}u=-\left\vert x\right\vert ^{a}u^{Q},
\]
which is a kind of \underline{\textit{unification of the two problems}}, with
source terms or absorption terms. For the case of source term ($\varepsilon
_{1}=1),$ we find again the results of \cite{B}, \cite{GuVe}, showing that the
new dynamical approach is simple and does not need regularity results or
energy functions. Moreover it gives a model for the study of system $(G)$.
Indeed if $p=q$, $a=b$ and $\delta+s=\mu+m,$ system $(G)$ admits solutions of
the form $(u,u),$ where $u$ is a solution of (\ref{one}) with $Q=\delta
+s.\medskip$

In the sequel of the article we study the case of source terms, i.e. $\left(
G\right)  =(S),$ where
\begin{equation}
(S)\left\{
\begin{array}
[c]{c}%
-\Delta_{p}u=\left\vert x\right\vert ^{a}u^{s}v^{\delta},\\
-\Delta_{q}v=\left\vert x\right\vert ^{b}u^{\mu}v^{m}.
\end{array}
\right.  \label{S}%
\end{equation}
This system has been studied by many authors, in particular the Hamiltonian
problem $s=m=0,$ in the linear case $p=q=2,$ see for example \cite{HuVo},
\cite{SZ1}, \cite{SZ2}, \cite{BuMa}, \cite{Sou}, \cite{deFPeRo}, and the
potential system where $\delta=m+1$, $\mu=s+1$ and $a=b,$ see \cite{BR},
\cite{TV}, \cite{TV2}; the problem with general powers has been studied in
\cite{B1}, \cite{Z3}, \cite{Z4}, \cite{Z5} in the linear case and \cite{BPo},
\cite{CFMiT}, \cite{Z2} in the quasilinear case, see also \cite{AzCM},
\cite{ChLuG}, \cite{CGGM}.\medskip\ 

Here we suppose that $\delta,\mu>0$, so that the system is always coupled,
$s,m\geqq0,$ and we assume for simplicity%
\begin{equation}
1<p,q<N,\text{ \qquad}\min(p+a,q+b)>0,\text{\qquad}D=\delta\mu
-(p-1-s)(q-1-m)>0. \label{ht}%
\end{equation}
We say that a positive solution $(u,v)$ in $(0,R)$ is \underline
{\textit{regular}} at $0$ if $u,v$ $\in C^{2}\left(  0,R\right)  \cap
C(\left[  0,R)\right)  $. Condition $\min(p+a,q+b)>0$ guaranties the existence
of local regular solutions. Then $u,v\in C^{1}(\left[  0,R)\right)  $. when
$a,b>-1,$ and $u^{\prime}(0)=v^{\prime}(0)=0$. The assumption $D>0$ is a
classical condition of superlinearity for the system. \medskip

We are interessed in the existence or nonexistence of \underline
{\textit{ground states}}, called G.S$.$, that means global positive $(u,v)$ in
$\left(  0,\infty\right)  $ and regular at $0.$ We exclude the case of
"trivial" solutions, $(u,v)=\left(  0,C\right)  $ or $\left(  C,0\right)  ,$
where $C$ is a constant, which can exist when $s>0$ or $m>0.$ \medskip

In Section \ref{local} we give a series of \underline{\textit{local existence
or nonexistence}} results concerning system $(S)$, which complete the
nonexistence results found in the litterature. They are not based on the fixed
point method, quite hard in general, see for example \cite{GuVe}, \cite{R}. We
make a dynamical analysis of the linearization of system $(M)$ near each fixed
point, which appears to be performant, even for the regular solutions. For a
better exposition, the proofs are given at Section \ref{app}.$\medskip$

In Section \ref{glo} we study the \underline{\textit{global existence}} of
G.S. This problem has been often compared with the nonexistence of positive
solutions of the Dirichlet problem in a ball, see \cite{SZ2}, \cite{SZ3},
\cite{CFMiT}, \cite{CGGM}. Here we use a shooting method adapted to system
$(M)$, which allows to avoid questions of regularity of system $(S).$ We give
a new way of comparison, and improve the former results:

\begin{theorem}
\label{equi}(i) Assume $s<\frac{N(p-1)+p+pa}{N-p}$ and $m<\frac{N(q-1)+q+qb}%
{N-q}.$ If system $(S)$ has no G.S., then\medskip

(i) there exist regular radial solutions such that $X(T)=\frac{N-p}{p-1}$ and
$Y(T)=\frac{N-q}{q-1}$ for some $T>0,$ with $0<X<\frac{N-p}{p-1}$ and
$0<Y<\frac{N-q}{q-1}$ on $(-\infty,T).$\medskip

(ii) there exists a positive radial solution $(u,v)$ of the Dirichlet problem
in a ball $B(0,R)$.\medskip
\end{theorem}

This result is a key tool in the next Sections for proving the existence of a
G.S. It gives also new existence results for the Dirichlet problem, see
Corollary \ref{dir}. We also give a complementary result:

\begin{proposition}
\label{com} Assume $s\geqq\frac{N(p-1)+p+pa}{N-p}$ and $m\geqq\frac
{N(q-1)+q+qb}{N-q}.$ Then all the regular radial solutions are $G.S.$\medskip
\end{proposition}

In Section \ref{HS} we study the \underline{\textit{radial}} solutions of the
well known Hamiltonian system
\[
(SH)\left\{
\begin{array}
[c]{c}%
-\Delta u=\left\vert x\right\vert ^{a}v^{\delta},\\
-\Delta v=\left\vert x\right\vert ^{b}u^{\mu},
\end{array}
\right.
\]
corresponding to $p=q=2<N,$ $s=m=0,$ $a>-2,$ which is variational. In the case
$a=b=0,$ a main conjecture was made in \cite{SZ4}:

\begin{conjecture}
\label{(C)} System $(SH)$ with $a=b=0$ admits no (radial or nonradial) G.S. if
and only if $(\delta,\mu)$ is under the hyperbola of equation
\[
\frac{N}{\delta+1}+\frac{N}{\mu+1}=N-2.
\]

\end{conjecture}

\noindent The question is still open; it was solved in the radial case in
\cite{Mi}, \cite{SZ2}, then partially in \cite{SZ1}, \cite{BuMa}, and up to
the dimension $N=4$ in \cite{Sou}, see references therein. Here we find again
and extend to the case $a,b\neq0$ some results of \cite{HuVo} relative to the
G.S., with a shorter proof. We also give an existence result for the Dirichlet
problem improving a result of \cite{deFPeRo}.

\begin{theorem}
\label{cla}Let $\mathcal{H}_{0}$ be the critical hyperbola in the plane
$(\delta,\mu)$ defined by
\begin{equation}
\frac{N+a}{\delta+1}+\frac{N+b}{\mu+1}=N-2. \label{ero}%
\end{equation}
Then

(i) System $(SH)$ admits a (unique) radial G.S. if and only if ($\delta,\mu)$
is above $\mathcal{H}_{0}$ or on $\mathcal{H}_{0}$. \medskip

(ii) The radial Dirichlet problem in a ball has a solution if and only if
($\delta,\mu)$ is under $\mathcal{H}_{0}$.\medskip\ 

(iii) On $\mathcal{H}_{0}$ the G.S. has the following behaviour at $\infty:$
assuming for example $\delta>\frac{N+a}{N-2},$ then $\lim_{r\rightarrow\infty
}r^{N-2}u(r)=\alpha>0,$ and%
\begin{align*}
\lim_{r\rightarrow\infty}r^{(N-2)\mu-(2+b)}v  &  =\beta>0\qquad\text{if }%
\mu<\frac{N+b}{N-2},\\
\lim_{r\rightarrow\infty}r^{N-2}v  &  =\beta>0\qquad\text{if }\mu>\frac
{N+b}{N-2},\\
\lim_{r\rightarrow\infty}r^{N-2}\left\vert \ln r\right\vert ^{-1}v  &
=\beta>0\qquad\text{if }\mu=\frac{N+b}{N-2}.
\end{align*}

\end{theorem}

Our proofs use a Pohozaev type function; in terms of the new variables
$X,Y,Z,W$, it contains a quadratic factor%
\begin{align}
\mathcal{E}_{H}(r)  &  =r^{N}\left[  u^{\prime}v^{\prime}+r^{b}\frac
{\left\vert u\right\vert ^{\mu+1}}{\mu+1}+r^{a}\frac{\left\vert v\right\vert
^{\delta+1}}{\delta+1}+\frac{N+a}{\delta+1}\frac{vu^{\prime}}{r}+\frac
{N+b}{\mu+1}\frac{uv^{\prime}}{r}\right] \nonumber\\
&  =r^{N-2}uv\left[  XY-\frac{Y(N+b-W)}{\mu+1}-\frac{(N+a-Z)X}{\delta
+1}\right]  . \label{Pom}%
\end{align}
As observed in (\cite{HuVo}) the G.S. can present a non-symmetric behaviour.
This non-symmetry phenomena has to be taken in account for solving conjecture
(\ref{(C)}).\medskip

In Section \ref{NV} we consider the radial solutions of a \underline
{\textit{nonvariational system}}:%

\[
(SN)\left\{
\begin{array}
[c]{c}%
-\Delta u=\left\vert x\right\vert ^{a}u^{s}v^{\delta},\\
-\Delta v=\left\vert x\right\vert ^{a}u^{\mu}v^{s},
\end{array}
\right.
\]
where $p=q=2<N,$ $a=b>-2$ and $m=s>0.$ For small $s$ it appears as a
perturbation of system $(SH).$ In the litterature very few results are known
for such nonvariational systems. Our main result in this Section is a new
result of \underline{\textit{existence of G.S}}. valid for any $s$:

\begin{theorem}
\label{cru}Consider the system $(SN),$ with $N>2,$ $a>-2.$ We define a curve
$\mathcal{C}_{s}$ in the plane $(\delta,\mu)$ by
\begin{equation}
\text{ }\frac{N+a}{\mu+1}+\frac{N+a}{\delta+1}=N-2+\frac{(N-2)s}{2}\min
(\frac{1}{\mu+1},\frac{1}{\delta+1}), \label{nc}%
\end{equation}
located under the hyperbola defined by (\ref{ero})$.$ If $(\delta,\mu)$ is
above $\mathcal{C}_{s},$ system $(SN)$ admits a G.S.\medskip
\end{theorem}

\noindent This result is obtained by constructing a new type of energy
function which contains two terms in $X^{2},Y^{2}$ :%
\begin{align}
\Phi(r)  &  =r^{N}\left[  u^{\prime}v^{\prime}+r^{b}\frac{u^{\mu+1}v^{s}}%
{\mu+1}+r^{a}\frac{u^{s}v^{\delta+1}}{\delta+1}+\frac{N+a}{\delta+1}%
\frac{vu^{\prime}}{r}+\frac{N+b}{\mu+1}\frac{uv^{\prime}}{r}+\frac{s}%
{2(\delta+1)}\frac{vu^{\prime2}}{u}+\frac{s}{2(\mu+1)}\frac{uv^{\prime2}}%
{v}\right] \nonumber\\
&  =r^{N-2}uv\left[  XY-\frac{Y(N+b-W)}{\mu+1}-\frac{(N+a-Z)X}{\delta+1}%
+\frac{s}{2(\delta+1)}X^{2}+\frac{s}{2(\mu+1)}Y^{2}\right]  . \label{Pon}%
\end{align}
\bigskip

In Section \ref{RP} we consider the \underline{\textit{radial}} solutions of
the potential system%
\[
(SP)\left\{
\begin{array}
[c]{c}%
-\Delta_{p}u=\left\vert x\right\vert ^{a}u^{s}v^{m+1},\\
-\Delta_{q}v=\left\vert x\right\vert ^{a}u^{s+1}v^{m},
\end{array}
\right.
\]
where $\delta=m+1,\mu=s+1$ and $a=b,$ which is variational, see \cite{TV},
\cite{TV2}. Using system ($M)$ we deduce new results of existence:

\begin{theorem}
\label{clo}Let $\mathcal{D}$ be the critical line in the plane $(m,s)$ defined
by
\[
N+a=(m+1)\frac{N-q}{q}+\left(  s+1\right)  \frac{N-p}{p}.
\]
Then

(i) System $(SP)$ admits a radial G.S. if and only if ($m,s)$ is above or on
$\mathcal{D}.\medskip$

(ii) On $\mathcal{D}$ the G.S. has the following behaviour: suppose for
example $q\leqq p.$ Let $\lambda^{\ast}=N+a-(s+1)\frac{N-p}{p-1}-m\frac
{N-q}{q-1}.$ Then $\lim_{r\rightarrow\infty}r^{\frac{N-p}{p-1}}u(r)=\alpha>0,$
and
\begin{align}
\lim_{r\rightarrow\infty}r^{\frac{N-q}{q-1}}v(r)  &  =\beta>0\qquad\text{if
}\lambda^{\ast}<0,\label{iu}\\
\lim_{r\rightarrow\infty}r^{\frac{\frac{N-p}{p-1}\mu-(q+b)}{q-1-m}}v(r)  &
=\beta>0\qquad\text{if }\lambda^{\ast}>0,\label{id}\\
\lim_{r\rightarrow\infty}r^{\frac{N-q}{q-1}}\left\vert \ln r\right\vert
^{-\frac{1}{q-1-m}}v(r)  &  =\beta>0\qquad\text{if }\lambda^{\ast}=0.
\label{it}%
\end{align}
In particular (\ref{iu}) holds if $p=q,$ or $q\leqq m+1.$\medskip\ 

(iii) The radial Dirichlet problem in a ball has a solution if and only if
($m,s)$ is under $\mathcal{D}$.\medskip
\end{theorem}

In that case we use the following energy function, which deserves to be
compared with the one of Section \ref{HS} , since it has also a quadratic
factor:%
\[
\mathcal{E}_{P}(r)=r^{N}\left[  (s+1)(\frac{\left\vert u^{\prime}\right\vert
^{p}}{p^{\prime}}+\frac{N-p}{p}\frac{u\left\vert u^{\prime}\right\vert
^{p-2}u^{\prime}}{r})+(m+1)(\frac{\left\vert v^{\prime}\right\vert ^{q}%
}{q^{\prime}}+\frac{N-q}{q}\frac{v\left\vert v^{\prime}\right\vert
^{q-2}v^{\prime}}{r})+r^{a}u^{s+1}v^{m+1}\right]
\]%
\begin{equation}
=r^{N-2-a}\frac{\left\vert u^{\prime}\right\vert ^{p-1}\left\vert v^{\prime
}\right\vert ^{q-1}}{u^{s}v^{m}}\left[  ZW-\frac{(s+1)W(N-p-(p-1)X)}{p}%
-\frac{(m+1)Z(N-q-(q-1)Y)}{q}\right]  . \label{ppo}%
\end{equation}
\bigskip

Finally in Section \ref{nonradial} we deduce a \underline{\textit{nonradial}}
result for the potential system in the case of two Laplacians:%
\[
(SL)\left\{
\begin{array}
[c]{c}%
-\Delta u=\left\vert x\right\vert ^{a}u^{s}v^{m+1},\\
-\Delta v=\left\vert x\right\vert ^{a}u^{s+1}v^{m}.
\end{array}
\right.
\]
Our result proves a conjecture proposed in \cite{BR}, showing that in the
subcritical case there exists no G.S.:

\begin{theorem}
\label{solve}Assume $a>-2$ and $s,m\geqq0.$ If
\begin{equation}
s+m+1<\min(\frac{N+2}{N-2},\frac{N+2+2a}{N-2}), \label{naj}%
\end{equation}
then system $(SL)$ admits no (radial or nonradial) G.S.\medskip
\end{theorem}

Our proof uses the estimates of \cite{BR}, which up to now are the only
extensions of the results of \cite{GiSp} to systems. It is based on the
construction of a nonradial Pohozaev function extending the radial one given
at (\ref{ppo}) for $p=q=2$, different from the energy function used in
\cite{BR}. \medskip

The case of the system $(G)$ with absorption terms ($\varepsilon
_{1}=\varepsilon_{2}=-1)$ or mixed terms ($\varepsilon_{1}=-\varepsilon
_{2}=1)$, studied in \cite{BG}, \cite{BGm}, will be the subject of a second
article. Our approach also extends to a system with gradient terms,
\begin{equation}
\left\{
\begin{array}
[c]{c}%
-\Delta_{p}u=\varepsilon_{1}\left\vert x\right\vert ^{a}u^{s}v^{\delta
}\left\vert \nabla u\right\vert ^{\eta}\left\vert \nabla v\right\vert ^{\ell
},\\
-\Delta_{q}v=\varepsilon_{2}\left\vert x\right\vert ^{b}u^{\mu}v^{m}\left\vert
\nabla u\right\vert ^{\nu}\left\vert \nabla v\right\vert ^{\kappa},
\end{array}
\right.  \label{G}%
\end{equation}
which will be studied in another work.\medskip

\textbf{Acknoledgment} The authors are grateful to Raul Manasevich whose
stimulating discussions encouraged us to study system $(G)$.

\section{Reduction to a quadratic system\label{sys}}

\subsection{The change of unknowns}

Here we consider the radial positive solutions $r\mapsto(u(r),v(r))$ of system
$(G)$ on any interval $\left(  R_{1},R_{2}\right)  $, that means
\[
\left\{
\begin{array}
[c]{c}%
\left(  \left\vert u^{\prime}\right\vert ^{p-2}u^{\prime}\right)  ^{\prime
}+\frac{N-1}{r}\left\vert u^{\prime}\right\vert ^{p-2}u^{\prime}%
=r^{1-N}\left(  r^{N-1}\left\vert u^{\prime}\right\vert ^{p-2}u^{\prime
}\right)  ^{\prime}=-\varepsilon_{1}r^{a}u^{s}v^{\delta},\\
\left(  \left\vert v^{\prime}\right\vert ^{q-2}v^{\prime}\right)  ^{\prime
}+\frac{N-1}{r}\left\vert v^{\prime}\right\vert ^{q-2}v^{\prime}%
=r^{1-N}\left(  r^{N-1}\left\vert v^{\prime}\right\vert ^{p-2}v^{\prime
}\right)  ^{\prime}=-\varepsilon_{2}r^{b}u^{\mu}v^{m}.
\end{array}
\right.
\]
Near any point $r$ where $u(r)\neq0,u^{\prime}(r)\neq0$ and $v(r)\neq0,$
$v^{\prime}(r)\neq0$ we define
\begin{equation}
X(t)=-\frac{ru^{\prime}}{u},\quad Y(t)=-\frac{rv^{\prime}}{v},\quad
Z(t)=-\varepsilon_{1}r^{1+a}u^{s}v^{\delta}\left\vert u^{\prime}\right\vert
^{-p}u^{\prime},\quad W(t)=-\varepsilon_{2}r^{1+b}u^{\mu}v^{m}\left\vert
v^{\prime}\right\vert ^{-q}v^{\prime}, \label{xyt}%
\end{equation}
where $t=\ln r.$ Then we find the system%
\[
(M)\left\{
\begin{array}
[c]{c}%
X_{t}=X\left[  X-\frac{N-p}{p-1}+\frac{Z}{p-1}\right]  ,\\
Y_{t}=Y\left[  Y-\frac{N-q}{q-1}+\frac{W}{q-1}\right]  ,\\
Z_{t}=Z\left[  N+a-sX-\delta Y-Z\right]  ,\\
W_{t}=W\left[  N+b-\mu X-mY-W\right]  .
\end{array}
\right.
\]
This sytem is \underline{\textit{quadratic}}, and moreover a very simple one,
of \underline{\textit{Kolmogorov type}}: it admits four invariant hyperplanes:
$X=0,Y=0,Z=0,W=0$. As a first consequence all the fixed points of the system
are explicite. The trajectories located on these hyperplanes do not correspond
to a solution of system $(G);$ they will be called \textit{\underline
{\textit{nonadmissible}}}. \medskip

We suppose that the discriminant of the system
\begin{equation}
D=\delta\mu-(p-1-s)(q-1-m)\neq0. \label{dis}%
\end{equation}
Then one can express $u,v$ in terms of the new variables:
\begin{equation}
u\mathbf{=}r^{-\gamma}\mathbf{(}\left\vert X\right\vert ^{p-1}\left\vert
Z\right\vert )^{(q-1-m)/D}\mathbf{(}\left\vert Y\right\vert ^{q-1}\left\vert
W\right\vert )^{\delta/D},\mathbf{\qquad}v\mathbf{=}r^{-\xi}\mathbf{(}%
\left\vert X\right\vert ^{p-1}\left\vert Z\right\vert )^{\mu/D}\mathbf{(}%
\left\vert Y\right\vert ^{q-1}\left\vert W\right\vert )^{(p-1-s)/D},
\label{exuv}%
\end{equation}
where $\gamma$ and $\xi$ are defined by
\begin{equation}
\gamma=\frac{(p+a)(q-1-m)+(q+b)\delta}{D},\qquad\xi=\frac
{(q+b)(p-1-s)+(p+a)\mu}{D}, \label{ga}%
\end{equation}
or equivalently by
\begin{equation}
(p-1-s)\gamma+p+a=\delta\xi,\qquad(q-1-m)\xi+q+b=\mu\gamma. \label{gk}%
\end{equation}
Since system $(M)$ is autonomous, each admissible trajectory $\mathcal{T}$ in
the phase space corresponds to a solution $(u,v)$ of system $(G)$ unique up to
a scaling: if $(u,v)$ is a solution, then for any $\theta>0$, $r\mapsto
(\theta^{\gamma}u(\theta r),\theta^{\xi}v(\theta r))$ is also a solution.

\subsection{Fixed points of system (M)}

System $(M)$ has at most $16$ fixed points. The main fixed point is%

\begin{equation}
M_{0}=\left(  X_{0},Y_{0},Z_{0},W_{0}\right)  =\left(  \gamma,\xi
,N-p-(p-1)\gamma,N-q-(q-1)\xi\right)  , \label{FIXE}%
\end{equation}
corresponding to the particular solutions
\begin{equation}
u_{0}(r)=Ar^{-\gamma},v_{0}(r)=Br^{-\xi},\quad A,B>0, \label{zero}%
\end{equation}
when they exist, depending on $\varepsilon_{1},\varepsilon_{2}.$ The values of
$A$ and $B$ are given by
\begin{align*}
A^{D}  &  =\left(  \varepsilon_{1}\gamma^{p-1}(N-p-\gamma(p-1))\right)
^{q-1-m}\left(  \varepsilon_{2}\xi^{q-1}(N-q-(q-1)\xi)\right)  ^{\delta},\\
B^{D}  &  =\left(  \varepsilon_{2}\xi^{q-1}(N-q-(q-1)\xi)\right)
^{p-1-s}\left(  \varepsilon_{1}\gamma^{p-1}(N-p-(p-1)\gamma)\right)  ^{\mu}.
\end{align*}
The other fixed points are%
\begin{align*}
0  &  =(0,0,0,0),\quad N_{0}=(0,0,N+a,N+b),\quad A_{0}=(\frac{N-p}{p-1}%
,\frac{N-q}{q-1},0,0),\\
I_{0}  &  =(\frac{N-p}{p-1},0,0,0),\quad J_{0}=(0,\frac{N-q}{q-1},0,0),\quad
K_{0}=(0,0,N+a,0),\quad L_{0}=(0,0,0,N+b),\\
G_{0}  &  =(\frac{N-p}{p-1},0,0,N+b-\frac{N-p}{p-1}\mu),\quad H_{0}%
=(0,\frac{N-q}{q-1},N+a-\frac{N-q}{q-1}\delta,0),
\end{align*}
\textbf{ }and if $m\neq q-1,$%
\begin{align*}
P_{0}  &  =(\frac{N-p}{p-1},\frac{\frac{N-p}{p-1}\mu-(q+b)}{q-1-m}%
,0,\frac{(q-1)(N+b-\frac{N-p}{p-1}\mu)-m(N-q)}{q-1-m}),\\
C_{0}  &  =\left(  0,-\frac{q+b}{q-1-m},0,\frac{(N+b)(q-1)-m(N-q)}%
{q-1-m}\right)  ,\\
R_{0}  &  =\left(  0,-\frac{q+b}{q-1-m},N+a+\delta\frac{b+q}{q-1-m}%
,\frac{(N+b)(q-1)-m(N-q)}{q-1-m}\right)  ,
\end{align*}
and by symmetry, if $s\neq p-1,$%
\begin{align*}
Q_{0}  &  =(\frac{\frac{N-q}{q-1}\delta-(p+a)}{p-1-s},\frac{N-q}{q-1}%
,\frac{(p-1)(N+a-\frac{N-q}{q-1}\delta)-s(N-p)}{p-1-s},0),\\
D_{0}  &  =\left(  -\frac{p+a}{p-1-s},0,\frac{(N+a)(p-1)-s(N-p)}%
{p-1-s},0\right)  ,\\
S_{0}  &  =\left(  -\frac{p+a}{p-1-s},0,\frac{(N+a)(p-1)-s(N-p)}%
{p-1-s},N+b+\mu\frac{a+p}{p-1-s}\right)  .
\end{align*}

\subsection{First comments}

\begin{remark}
This formulation allows to treat more general systems with signed solutions by
reducing the study on intervals where $u$ and $v$ are nonzero. Consider for
example the problem%
\[
-\Delta_{p}u=\varepsilon_{1}\left\vert x\right\vert ^{a}\left\vert
u\right\vert ^{s}\left\vert v\right\vert ^{\delta-1}v,\quad-\Delta
_{q}v=\varepsilon_{2}\left\vert x\right\vert ^{b}\left\vert v\right\vert
^{m}\left\vert u\right\vert ^{\mu-1}u.
\]
0n any interval where $uv>0,$ the couple $(\left\vert u\right\vert ,\left\vert
v\right\vert )$ is a solution of $(G).$ On any interval where $u>0>v,$ the
couple $(u,\left\vert v\right\vert )$ satisfies $(G)$ with $(\varepsilon
_{1},\varepsilon_{2})$ replaced by $(-\varepsilon_{1},-\varepsilon_{2}%
).$\medskip
\end{remark}

\begin{remark}
\label{vauv}There is another way for reducing the system to an autonomous
form: setting
\[
\mathrm{U}(t)=r^{\gamma}u,\quad\mathrm{V}(t)=r^{\xi}v,\quad\mathrm{H}%
(t)=-r^{\left(  \gamma+1\right)  (p-1)}\left\vert u^{\prime}\right\vert
^{p-2}u^{\prime},\quad\mathrm{K}(t)=-r^{\left(  \xi+1\right)  (q-1)}\left\vert
v^{\prime}\right\vert ^{q-2}v^{\prime},
\]
with $t=\ln r,$ we find
\begin{equation}
\left\{
\begin{array}
[c]{c}%
\mathrm{U}_{t}=\gamma\mathrm{U}-\left\vert \mathrm{H}\right\vert
^{(2-p)/(p-1)}\mathrm{H}\mathbf{,\qquad}\mathrm{V}_{t}=\zeta\mathrm{U}%
-\left\vert \mathrm{K}\right\vert ^{(2-q)/(q-1)}\mathrm{K}\mathbf{,}\\
\mathrm{H}_{t}=(\gamma(p-1)+p-N)\mathrm{H}+\varepsilon_{1}\mathrm{U}%
^{s}\mathrm{V}^{\delta},\mathbf{\qquad}\mathrm{K}_{t}=(\zeta
(q-1)+q-N)\mathrm{K}+\varepsilon_{2}\mathrm{U}^{\mu}\mathrm{V}^{m}.
\end{array}
\right.  \label{deh}%
\end{equation}
It extends the well-known transformation of Emden-Fowler in the scalar case
when $p=2,$ used also in \cite{B} for general $p,$ see Section \ref{scal}.
When $p=q=2$ we obtain
\begin{equation}
\left\{
\begin{array}
[c]{c}%
\mathrm{U}_{tt}+(N-2-2\gamma)\mathrm{U}_{t}-\gamma(N-2-\gamma)\mathrm{U}%
+\varepsilon_{1}\mathrm{U}^{s}\mathrm{V}^{\delta}=0,\\
\mathrm{V}_{tt}+(N-2-2\xi)\mathrm{V}_{t}-\xi(N-2-\xi)\mathrm{V}+\varepsilon
_{2}\mathrm{U}^{\mu}\mathrm{V}^{m}=0,
\end{array}
\right.  \label{uvs}%
\end{equation}
which was extended to the nonradial case and used for Hamiltonian systems
($s=m=0),$ with source terms in \cite{BuMa} ($\varepsilon_{1}=\varepsilon
_{2}=1)$ and absorption terms in \cite{BG} ($\varepsilon_{1}=\varepsilon
_{2}=-1)$. Our system is more adequated for finding the possible behaviours:
unlike system (\ref{deh})it has no singularity, since it is polynomial, also
its fixed points at $\infty$ are not concerned when we deal with solutions
$u,v>0.$\medskip
\end{remark}

\begin{remark}
\label{hul}In the specific case $p=q=2,$ setting
\[
z=XZ=\varepsilon_{1}r^{2+a}\left\vert u\right\vert ^{s-2}u\left\vert
v\right\vert ^{\delta-1}v,\qquad w=YW=\varepsilon_{2}r^{2+b}\left\vert
u\right\vert ^{\mu-1}u\left\vert v\right\vert ^{m-2}v,
\]
we get the following system%
\[
\left\{
\begin{array}
[c]{c}%
X_{t}=X^{2}-(N-2)X+z,\mathbf{\qquad}Y_{t}=Y^{2}-(N-2)Y+w,\\
z_{t}=z\left[  2+a+(1-s)X-\delta Y\right]  ,\mathbf{\qquad}w_{t}=w\left[
2+b-\mu X+(1-m)Y\right]  .
\end{array}
\right.
\]
It has been used in \cite{HuVo} for studying the Hamiltonian system $(SH)$.
Even in that case we will show at Section \ref{HS} that system $(M)$ is more
performant, because it is of Kolmogorov type.\medskip
\end{remark}

\begin{remark}
\label{duc} Assume\textbf{ }$p=q$\textbf{ }and\textbf{ }$a=b.$ Setting
$t=k\hat{t}$ and $\left(  \hat{X},\hat{Y},\hat{Z},\hat{W}\right)
=k(X,Y,Z,W)$, we obtain a system of the same type with $N,a$ replaced by
$\hat{N},\hat{a},$ with
\[
\frac{\hat{N}-p}{N-p}=k=\frac{\hat{N}+\hat{a}}{N+a}.
\]
It corresponds to the change of unknowns
\[
r=\hat{r}^{k},\qquad\hat{u}(\hat{r})=C_{1}u(r),\quad\hat{v}(\hat{r}%
)=C_{2}u(r),\qquad C_{1}=k^{p(p-1-m+\delta)/D},\quad C_{2}=k^{(p(p-1-s+\mu
))/D}.
\]
From (\ref{exuv}) and (\ref{ga}), we get $\hat{\gamma}/\gamma=\hat{\xi}%
/\xi=k=\frac{p+\hat{a}}{p+a}.$ There is one free parameter. In particular

1) we get a system without power ($\hat{a}=0),$ by taking
\[
\hat{N}=\frac{p(N+a)}{p+a},\qquad k=\frac{p}{p+a};
\]

2) we get a system in dimension $\hat{N}=1,$ by taking
\[
k=-\frac{p-1}{N-p}<0,,\qquad\hat{a}=\frac{p+a-(N+a)p}{N-p}.
\]

\end{remark}

\section{The scalar case\label{scal}}

We first study the signed solutions of two scalar equations with source or
absorption:
\begin{equation}
-\Delta_{p}u=-r^{1-N}\left(  r^{N-1}\left\vert u^{\prime}\right\vert
^{p-2}u^{\prime}\right)  ^{\prime}=\varepsilon\left\vert x\right\vert
^{a}\left\vert u\right\vert ^{Q-1}u, \label{sca}%
\end{equation}
with $\varepsilon=\pm1,$ $1<p<N,$ $Q\neq p-1$ and $p+a>0.\medskip$

We cannot quote all the huge litterature concerning its solutions,
supersolutions or subsolutions, from the first studies of Emden and Fowler for
$p=2,$ recalled in \cite{EF}; see for example \cite{B} and \cite{V}, for any
$p>1,$ and references therein. We set
\[
Q_{1}=\frac{(N+a)(p-1)}{N-p},\qquad Q_{2}=\frac{N(p-1)+p+pa}{N-p},\qquad
\gamma=\frac{p+a}{Q+1-p}.
\]
From Remark \ref{duc} we could reduce the system to the case $a=0$, in
dimension $\hat{N}=p(N+a)/(p+a).$ However we do not make the reduction,
because we are motivated by the study of system $(G)$, and also by the
nonradial case.

\subsection{A common phase plane for the two equations}

Near any point $r$ where $u(r)\neq0$ (positive or negative), and $u^{\prime
}(r)\neq0$ setting
\begin{equation}
X(t)=-\frac{ru^{\prime}}{u},\qquad Z(t)=-\varepsilon r^{1+a}\left\vert
u\right\vert ^{Q-1}u\left\vert u^{\prime}\right\vert ^{-p}u^{\prime},
\label{xz}%
\end{equation}
with $t=\ln r,$ we get a 2-dimensional system
\[
(M_{scal})\left\{
\begin{array}
[c]{c}%
X_{t}=X\left[  X-\frac{N-p}{p-1}+\frac{Z}{p-1}\right]  ,\\
Z_{t}=Z\left[  N+a-QX-Z\right]  .
\end{array}
\right.
\]
and then $\left\vert u\right\vert \mathbf{=}r^{-\gamma}\mathbf{(}\left\vert
Z\right\vert \left\vert X\right\vert ^{p-1})^{1/(Q+1-p)}.$ This change of
unknown was mentioned in \cite{ChiTi} in the case $p=2,\varepsilon=1$ and
$N=3.$ It is remarkable that system $(M_{scal})$ is the same for the
\textit{two cases }$\varepsilon=\pm1$, the only difference is that $X(t)Z(t)$
has the sign of $\varepsilon:\medskip$

The equation with source $(\varepsilon=1)$ is associated to the 1$^{st}$ and
$3^{rd}$ quadrant. It is well known that any local solution has a unique
extension on $\left(  0,\infty\right)  .$ The 1$^{st}$ quadrant corresponds to
the intervals where $\left\vert u\right\vert $ is decreasing, which can be of
the following types $\left(  0,\infty\right)  ,(0,R_{2})$,$\left(
R_{1},\infty\right)  $,$\left(  R_{1},R_{2}\right)  $, $0<R_{1}<R_{2}<\infty.$
The $3^{rd}$ quadrant corresponds to the intervals $\left(  R_{1}%
,R_{2}\right)  $ where $\left\vert u\right\vert $ is increasing.$\medskip$

The equation with absorption $(\varepsilon=-1)$ is associated to the $2^{nd}$
and $4^{th}$ quadrant. It is known that the solutions have at most one zero,
and their maximal interval of existence can be $(0,R_{2}),(R_{1}%
,\infty),\left(  R_{1},R_{2}\right)  $ or $(0,\infty).$ The $2^{nd}$ quadrant
corresponds to the intervals $\left(  R_{1},R_{2}\right)  $ where $\left\vert
u\right\vert $ is increasing. The $4^{th}$ quadrant corresponds to the
intervals $(0,R_{2})$ or $\left(  R_{1},\infty\right)  $ where $\left\vert
u\right\vert $ is decreasing.$\medskip$

The fixed points of $(M_{scal})$ are
\[
M_{0}=(X_{0},,Z_{0})=(\gamma,N-p-(p-1)\gamma),\quad(0,0),\quad N_{0}%
=(0,N+a),\quad A_{0}=(\frac{N-p}{p-1},0).
\]
In particular $M_{0}$ is in the 1$^{st}$ quadrant whenever $\gamma<\frac
{N-p}{p-1},$ equivalently $Q>Q_{1},$  and in the $4^{th}$ quadrant whenever
$Q<Q_{1}$. It corresponds to the solution
\[
u(r)=Ar^{-\gamma},\text{ \quad for }\varepsilon=1,Q>Q_{1},\quad\text{ or
}\varepsilon=-1,Q<Q_{1},
\]
where $A=\left(  \varepsilon\gamma^{p-1}(N-p-\gamma(p-1))\right)
^{1/(Q-p+1)}.$

\subsection{Local study}

We examine the fixed points, where for simplicity we suppose $Q\neq Q_{1},$
and we deduce local results for the two equations:$\medskip$

$\bullet$ Point $(0,0):$ it is a saddle point, and the only trajectories that
converge to $(0,0)$ are the separatrix, contained in the lines $X=0,Y=0,$ they
are not admissible.\medskip

$\bullet$ Point $N_{0}:$ it is a saddle point: the eigenvalues of the
linearized system are $\frac{p}{p-1}$ and $-N$. the trajectories ending at
$N_{0}$ at $\infty$ are located on the set $Z=0,$ then there exists a unique
trajectory starting from $-\infty$ at $N_{0}$; it corresponds to the local
existence and uniqueness of regular solutions, which we obtain easily.\medskip

$\bullet$ Point\textbf{ }$A_{0}:$ the eigenvalues of the linearized system are
$\frac{N-p}{p-1}$ and $\frac{N-p}{p-1}(Q_{1}-Q)$. If $Q<Q_{1},$ $A_{0}$ is an
unstable node. There is an infinity of trajectories starting from $A_{0}$ at
$-\infty;$ then $X(t)$ converges exponentially to $\frac{N-p}{p-1},$ thus
$\lim_{r\rightarrow0}$ $r^{\frac{N-p}{p-1}}u=\alpha>0.$ The corresponding
solutions $u$ satisfy the equation with a Dirac mass at $0.$ There exists no
solution converging to $A_{0}$ at $\infty.$ If $Q>Q_{1},$ $A_{0}$ is a saddle
point; the trajectories starting from $A_{0}$ at $-\infty$ are not
admissible$;$ there is a trajectory converging at $\infty,$ and then
$\lim_{r\rightarrow\infty}$ $r^{\frac{N-p}{p-1}}u=\alpha>0.$\medskip

$\bullet$ Point\textbf{ }$M_{0}:$ the eigenvalues $\lambda_{1},\lambda_{2}$ of
the linearized system are the roots of equation
\[
\lambda^{2}+(Z_{0}-X_{0})\lambda+\frac{Q-p+1}{p-1}X_{0}Z_{0}=0.
\]
For $\varepsilon=1$, $M_{0}$ is defined for $Q>Q_{1};$ the eigenvalues are
imaginary when $X_{0}=Z_{0},$ equivalently $\gamma=(N-p)/p,$ $Q=Q_{2}$. When
$Q<Q_{2},$ $M_{0}$ is a source, there exists an infinity of trajectories such
that $\lim_{r\rightarrow0}r^{\gamma}u=A$. When $Q>Q_{2},$ $M_{0}$ is a sink,
and there exists an infinity of trajectories such that $\lim_{r\rightarrow
\infty}r^{\gamma}u=A$. When $Q=Q_{2},$ $M_{0}$ is a center, from \cite{B} For
$\varepsilon=-1,$ $M_{0}$ is defined for $Q<Q_{1},$ it is a saddle-point.
There exist two trajectories $\mathcal{T}_{1},\mathcal{T}_{1}^{\prime}$
converging at $\infty,$ such that $\lim_{r\rightarrow\infty}r^{\gamma}u=A$ and
two trajectories $\mathcal{T}_{2},$ $\mathcal{T}_{2}^{\prime},$ converging at
$0,$ such that $\lim_{r\rightarrow0}r^{\gamma}u=A.$

\subsection{Global study}

\begin{remark}
System $(M_{scal})$ has no limit cycle for $Q\neq Q_{2}$. It is evident when
$\varepsilon=-1.$ When $\varepsilon=1,$ as noticed in \cite{GuVe}, it comes
from the Dulac's theorem: setting $X_{t}=f(X,Z),\quad Z_{t}=g(X,Z),$ and
\[
B(X,Z)=X^{pQ/(Q+1-p)-2}Z^{(p/(Q+1-p)-1},\quad M=B_{X}X_{t}+B_{Z}Z_{t}%
+B(f_{X}+g_{Z}),
\]
then $M=KB$ with $K=(Q_{2}-Q)\gamma(N-p)/p,$ thus $M$ has no zero for $Q\neq
Q_{2}.$\medskip
\end{remark}

Then from the Poincar\'{e}-Bendixson theorem, any trajectory bounded near
$\pm\infty$ converges to one of the fixed points. Thus we find again global
results: \medskip

$\bullet$ Equation with source $\left(  \varepsilon=1\right)  $. If $Q<Q_{1},$
there is no G.S.: the regular trajectory $\mathcal{T}$ issued from $N_{0}$
cannot converge to a fixed point. Then $X$ tends to $\infty$ and the regular
solutions $u$ are changing sign, there is no G.S..

If $Q_{1}<Q<Q_{2},$ the regular trajectory $\mathcal{T}$ cannot converge to
$M_{0};$ if it converges to $A_{0},$ it is the unique trajectory converging to
$A_{0}$; the set delimitated by $\mathcal{T}$ and $X=0,Z=0$ is invariant, thus
it contains $M_{0};$ and the trajectories issued from $M_{0}$ cannot converge
to a fixed point, which is contradictory. then again $X$ tends to $\infty$ on
$\mathcal{T}$ and the regular solutions $u$ are changing sign.. The trajectory
ending at $A_{0}$ converges to $M_{0}$ at $-\infty;$ then there exist
solutions $u>0$ such that $\lim_{r\rightarrow0}r^{\gamma}u=A$ and
$\lim_{r\rightarrow0}$ $r^{\frac{N-p}{p-1}}u=\alpha>0.$

If $Q>Q_{2},$ the only singular solution at $0$ is $u_{0},$ and the regular
solutions are G.S., with $\lim_{r\rightarrow\infty}r^{\gamma}u=A.$ Indeed
$M_{0}$ is a sink; the trajectory ending at $A_{0}$ cannot converge to $N_{0}$
at $-\infty$, thus $X$ converges to $0,$ and $Z$ converges to $\infty,$ then
$u$ cannot be positive on $(0,\infty).$The trajectory issued from $N_{0}$
converges to $M_{0}.$\medskip

$\bullet$ Equation with absorption $\left(  \varepsilon=-1\right)  $. If
$Q>Q_{1},$ all the solutions $u$ defined near $0$ are regular; indeed the
trajectories cannot converge to a fixed point.

If $Q<Q_{1},$  we find again easily a well known result: there exists a
positive solution $u_{1},$ unique up to a scaling, such that $\lim
_{r\rightarrow0}r^{\frac{N-p}{p-1}}u_{1}=\alpha>0,$ and $\lim_{r\rightarrow
\infty}r^{\gamma}u_{1}=A.$ Indeed the eigenvalues at $M_{0}$ satisfy
$\lambda_{1}<0<\lambda_{2}$. There are two trajectories $\mathcal{T}%
_{1},\mathcal{T}_{1}^{\prime}$ associated to $\lambda_{1},$ and the
eigenvector $(X_{0}+\left\vert \lambda_{1}\right\vert ,-\frac{X_{0}}{p-1}).$
The trajectory $\mathcal{T}_{1}$ satisfies $X_{t}>0>Z_{t}$ near $\infty,$ and
$X>\frac{N-p}{p-1},$ since $Z_{0}<0,$ and $X$ cannot take the value
$\frac{N-p}{p-1}$ because at such a point $X_{t}<0;$ then $\frac{N-p}%
{p-1}<X<X_{0}$ and $X_{t}>0$ as long as it is defined; similarly $Z_{0}<Z<0$
and $Z_{t}<0;$ then $\mathcal{T}_{1}$ converge to a fixed point, necessarily
$A_{0},$ showing the existence of $u_{1}.$ The trajectory $\mathcal{T}%
_{1}^{\prime}$ corresponds to solutions $u$ such that $\lim_{r\rightarrow
\infty}r^{\gamma}u=A$ and $\lim_{r\rightarrow R}u=\infty$ for some $R>0.$
There are two trajectories $\mathcal{T}_{2},$ $\mathcal{T}_{2}^{\prime},$
associated to $\lambda_{2},$ defining solutions $u$ such that $\lim
_{r\rightarrow0}r^{\gamma}u=A$ and changing sign, or with a minimum point and
$\lim_{r\rightarrow R}u=\infty$ for some $R>0.$ The regular trajectory starts
from $N_{0}$ in the $2^{nd}$ quadrant, it cannot converge to a fixed point,
then $\lim_{r\rightarrow R}u=\infty$ for some $R>0.$\medskip

$\bullet$ Critical case $Q=Q_{2}:$ it is remarkable that system $(M_{scal})$
admits \underline{\textit{another invariant line}}, namely $A_{0}N_{0},$ given
by
\begin{equation}
\frac{X}{p^{\prime}}+\frac{Z}{Q_{2}+1}-\frac{N-p}{p}=0.\label{line}%
\end{equation}
It precisely corresponds to well-known solutions of the two equations
\[
u=c(K^{2}+r^{(p+a)/(p-1)})^{(p-N)/(p+a)},\text{for }\varepsilon=1;\quad
u=c\left\vert K^{2}-r^{(p+a)/(p-1)}\right\vert ^{(p-N)/(p+a)},\text{for
}\varepsilon=-1,
\]
where $K^{2}=c^{Q-p+1}(N+a)^{-1}\left(  (N-p)/(p-1)\right)  ^{1-p}.$

\begin{remark}
\label{aut} The global results have been obtained \underline{without using
energy functions}\textbf{. } The study of \cite{B} was based on a reduction of
type of Remark \ref{vauv}, using an energy function linked to the new unknown.
Other energy functions are well-known, of Pohozaev type:
\[
\mathcal{F}_{\sigma}(r)=r^{N}\left[  \frac{\left\vert u^{\prime}\right\vert
^{p}}{p^{\prime}}+\varepsilon r^{a}\frac{\left\vert u\right\vert ^{Q+1}}%
{Q+1}+\sigma\frac{u\left\vert u^{\prime}\right\vert ^{p-2}u^{\prime}}%
{r}\right]  =r^{N-p}\left\vert u\right\vert ^{p}\left\vert X\right\vert
^{p-2}X\left[  \frac{X}{p^{\prime}}+\frac{Z}{Q+1}-\sigma\right]  ,
\]
with $\sigma=\frac{N-p}{p},$ satisfying $\mathcal{F}_{\sigma}^{\prime
}(r)=r^{N-1+a}\left(  \frac{N+a}{Q+1}-\frac{N-p}{p}\right)  \left\vert
u\right\vert ^{Q+1},$ or with $\sigma=\frac{N+a}{Q+1}$, leading to
$\mathcal{F}_{\sigma}^{\prime}(r)=r^{N-1}\left(  \frac{N+a}{Q+1}-\frac{N-p}%
{p}\right)  \left\vert u^{\prime}\right\vert ^{p}.$ In the critical case
$Q=Q_{2}$, all these functions coincide and they are constant, in other words
system $(M_{scal})$ has a first integral. We find again the line (\ref{line}):
the G.S. are the \underline{functions of energy\textbf{ }$0.$}
\end{remark}

\section{Local study of system $(S)$\label{local}}

In all the sequel we study the system with source terms: $(G)=(S)$. Assumption
(\ref{ht}) is the most interesting case for studying the existence of the G.S.
\medskip

We first study the local behaviour of nonnegative solutions $(u,v)$ defined
near $0$ or near $\infty.$ It is well known that any solution $(u,v)$ positive
on some interval $\left(  0,R\right)  $ satisfies $u^{\prime},v^{\prime}<0$ on
$\left(  0,R\right)  .$ Any solution $(u,v)$ positive on $(R,\infty),$
satisfies $u^{\prime},v^{\prime}<0$ near $\infty.$ We are reduced to study the
system in the region $\mathcal{R}$ where $X,Y,Z,W>0,$ and consider the fixed
points in $\mathcal{\bar{R}}.$ Then%
\begin{equation}
X(t)=-\frac{ru^{\prime}}{u},\quad Y(t)=-\frac{rv^{\prime}}{v},\quad
Z(t)=\frac{r^{1+a}u^{s}v^{\delta}}{\left\vert u^{\prime}\right\vert ^{p-1}%
},\quad W(t)=\frac{r^{1+b}v^{m}u^{\mu}}{\left\vert v^{\prime}\right\vert
^{q-1}}; \label{eto}%
\end{equation}
and ($X,Y,Z,W)$ is a solution of system $(M)$ in $\mathcal{R}$ if and only if
$(u,v)$ defined by
\begin{equation}
u\mathbf{=}r^{-\gamma}\mathbf{(}ZX^{p-1})^{(q-1-m)/D}\mathbf{(}WY^{q-1}%
)^{\delta/D},\mathbf{\qquad}v\mathbf{=}r^{-\xi}\mathbf{(}WY^{q-1}%
)^{(p-1-s)/D}\mathbf{(}ZX^{p-1})^{\mu/D} \label{for}%
\end{equation}
is a positive solution with $u^{\prime},v^{\prime}<0.$ Among the fixed points,
the point $M_{0}$ defined at (\ref{FIXE}) lies in $\mathcal{R}$ if and only
if
\begin{equation}
0<\gamma<\frac{N-p}{p-1}\quad\text{and \quad}0<\xi<\frac{N-q}{q-1}. \label{cd}%
\end{equation}
The local study of the system near $M_{0}$ appears to be tricky, see Remark
\ref{cet}. A main difference with the scalar case is that there always exist a
trajectory converging to $M_{0}$ at $\pm\infty:$

\begin{proposition}
\label{biz}(Point $M_{0})$ Assume that (\ref{cd}) holds. Then there exist
trajectories converging to $M_{0}$ as $r\rightarrow$ $\infty,$ and then
solutions $(u,v)$ being defined near $\infty,$ such that
\begin{equation}
\lim_{r\rightarrow\infty}r^{\gamma}u=\alpha>0,\quad\lim_{r\rightarrow\infty
}r^{\xi}v=\beta>0. \label{cov}%
\end{equation}
There exist trajectories converging to $M_{0}$ as $r\rightarrow0$, and thus
solutions $(u,v)$ being defined near $0$ such that
\begin{equation}
\lim_{r\rightarrow0}r^{\gamma}u=\alpha>0,\quad\lim_{r\rightarrow0}r^{\xi
}v=\beta>0. \label{cow}%
\end{equation}

\end{proposition}

\begin{proof}
Here $M_{0}$ $\in\mathcal{R};$ setting $X=X_{0}+\tilde{X},Y=Y_{0}+\tilde
{Y},Z=Z_{0}+\tilde{Z},W=W_{0}+\tilde{W},$ the linearized system is%
\[
\left\{
\begin{array}
[c]{c}%
\tilde{X}_{t}=X_{0}(\tilde{X}+\frac{1}{p-1}\tilde{Z}),\\
\tilde{Y}_{t}=Y_{0}(\tilde{Y}+\frac{1}{q-1}\tilde{W}),\\
\tilde{Z}_{t}=Z_{0}(-s\tilde{X}-\delta\tilde{Y}-\tilde{Z}),\\
\tilde{W}_{t}=W_{0}(-\mu\tilde{X}-m\tilde{Y}-\tilde{W}).
\end{array}
\right.
\]
The eigenvalues are the roots $\lambda_{1},\lambda_{2},\lambda_{3},\lambda
_{4},$ of equation
\begin{equation}
\left[  (\lambda-X_{0})(\lambda+Z_{0})+\frac{s}{p-1}X_{0}Z_{0}\right]  \left[
(\lambda-Y_{0})(\lambda+W_{0})+\frac{m}{q-1}Y_{0}W_{0}\right]  -\frac
{\delta\mu}{(p-1)(q-1)}X_{0}Y_{0}Z_{0}W_{0}=0. \label{valp}%
\end{equation}
This equation is of the form
\[
f(\lambda)=\lambda^{4}+E\lambda^{3}+F\lambda^{2}+G\lambda-H=0,
\]
with
\[
\left\{
\begin{array}
[c]{c}%
E=Z_{0}-X_{0}+W_{0}-Y_{0},\\
F=(Z_{0}-X_{0})(W_{0}-Y_{0})-\frac{s+p-1}{p-1}X_{0}Z_{0}-\frac{m+q-1}%
{q-1}Y_{0}W_{0},\\
G=-\frac{q-1-m}{q-1}Y_{0}W_{0}(Z_{0}-X_{0})-\frac{p-1-s}{p-1}X_{0}Z_{0}%
(W_{0}-Y_{0}),\\
H=\frac{D}{(p-1)(q-1)}X_{0}Y_{0}Z_{0}W_{0}.
\end{array}
\right.
\]
From (\ref{ht}) we have $H>0,$ then $\lambda_{1}\lambda_{2}\lambda_{3}%
\lambda_{4}<0.$ There exist two real roots $\lambda_{3}<0<\lambda_{4},$ and
two roots $\lambda_{1},\lambda_{2}$, real with $\lambda_{1}\lambda_{2}>0$, or
complex$.$ Therefore there exists at least one trajectory converging to
$M_{0}$ at $\infty$ and another one at $-\infty.$ Then (\ref{cov}) and
(\ref{cow}) follow from (\ref{for}). Moreover the convergence is monotone for
$X,Y,Z,W.$\medskip
\end{proof}

\begin{remark}
\label{cet}There exist imaginary roots, namely $\operatorname{Re}\lambda
_{1}=\operatorname{Re}\lambda_{2}=0,$ if and only if there exists $c>0$ such
that $f(ci)=0,$ that means $Ec^{2}-G=0,$ and $c^{4}-Fc^{2}-H=0,$ equivalently
\[
E=G=0,\quad\text{or }EG>0\text{ and }G^{2}-EFG-E^{2}H=0.
\]
Condition \textbf{ }$E=G=0$ means that\medskip

(i) either $Z_{0}=X_{0}$ and $W_{0}=Y_{0},$ i.e.
\begin{equation}
\left(  \gamma,\xi\right)  =(\frac{N-p}{p},\frac{N-q}{q}), \label{pre}%
\end{equation}
in other words $(\delta,\mu)=(\frac{q\left(  N(p-1-s)+p(1-s+a\right)
)}{p(N-q)},\frac{p\left(  N(q-1-m)+q(1-m+b\right)  )}{q(N-p)}).\medskip$

(ii) or $(p-1-s)(q-1-m)>0$ and $\left(  \gamma,\xi\right)  $ satisfies
\begin{equation}
\left\{
\begin{array}
[c]{c}%
2N-p-q=p\gamma+q\xi,\\
\text{ }(1-\frac{m}{q-1})\xi(N-q-(q-1)\xi)=(1-\frac{s}{p-1})\gamma
(N-p-(p-1)\gamma).
\end{array}
\right.  \label{cxd}%
\end{equation}
This gives in general 0,1 or 2 values of $\left(  \gamma,\xi\right)  $. For
example, in the case $\frac{m}{q-1}=\frac{s}{p-1}\neq1,$ and $(p-2)(q-2)>0$
and $N>\frac{pq-p-q}{p+q-2}$ we find another value, different from the one of
(\ref{pre}) for $p\neq q:$
\begin{equation}
(\gamma,\xi)=(N\frac{q-2}{pq-p-q}-1,N\frac{p-2}{pq-p-q}-1). \label{pri}%
\end{equation}
Moreover the computation shows that it can exist imaginary roots with
$E,G\neq0.$\medskip
\end{remark}

In the case $p=q=2$ and $s=m$ the situation is interesting:

\begin{proposition}
\label{cent} Assume $p=q=2$ and $s=m<\frac{N}{N-2},$ with $\delta
+1-s>0,\mu+1-s>0.$ In the plane $(\delta,\mu),$ let $\mathcal{H}_{s}$ be the
hyperbola of equation
\begin{equation}
\frac{1}{\delta+1-s}+\frac{1}{\mu+1-s}=\frac{N-2}{N-(N-2)s},\text{ }
\label{defi}%
\end{equation}
\textbf{ }equivalently $\gamma+\xi=N-2.$ Then $\mathcal{H}_{s}$ is contained
in the set\textbf{ }of points $(\delta,\mu)$ for which the linearized system
at $M_{0}$ has imaginary roots, and equal when $s\leqq1$.\medskip
\end{proposition}

\begin{proof}
The assumption $D>0$ imply $\delta+1-s>0$ and $\mu+1-s>0$; condition $E=G=0$
implies $s<N/(N-2)$ and reduces to condition (\ref{defi}). Moreover if
$s\leqq1,$ all the cases are covered. Indeed $2G=(s-1)E\left[  Y_{0}%
Z_{0}+X_{0}W_{0}\right]  ,$ hence $GE\leqq0.\medskip$
\end{proof}

Next we give a summary of the local existence results obtained by
linearization around the other fixed points of system $(M)$ proved in Section
\ref{app}. Recall that $t\rightarrow-\infty$ as $r\rightarrow0$ and
$t\rightarrow\infty$ as $r\rightarrow\infty.$

\begin{proposition}
\label{ereg} (Point $N_{0})$ A solution $(u,v)$ is regular if and only if the
corresponding trajectory converges to $N_{0}$ when $r\rightarrow0$. For any
$u_{0},v_{0}>0,$ there exists a unique local regular solution $(u,v)$ with
initial data $(u_{0},v_{0}).\medskip$
\end{proposition}

\begin{proposition}
\label{cao} (Point $A_{0})$ If $s\frac{N-p}{p-1}+\delta\frac{N-q}{q-1}>N+a$
and $\mu\frac{N-p}{p-1}+m\frac{N-q}{q-1}>N+b,$ there exist (admissible)
trajectories converging to $A_{0}$ when $r\rightarrow\infty$. If $s\frac
{N-p}{p-1}+\delta\frac{N-q}{q-1}<N+a$ and $\mu\frac{N-p}{p-1}+m\frac{N-q}%
{q-1}<N+b,$ the same happens when $r\rightarrow0$. In any case
\begin{equation}
\lim r^{\frac{N-p}{p-1}}u=\alpha>0,\quad\lim r^{\frac{N-q}{q-1}}v=\beta>0.
\label{oup}%
\end{equation}
If $s\frac{N-p}{p-1}+\delta\frac{N-q}{q-1}<N+a$ or $\mu\frac{N-p}{p-1}%
+m\frac{N-q}{q-1}<N+b,$ there exists no trajectory converging when
$r\rightarrow\infty;$ if $s\frac{N-p}{p-1}+\delta\frac{N-q}{q-1}>N+a$ or
$\mu\frac{N-p}{p-1}+m\frac{N-q}{q-1}>N+b,$ there exists no trajectory
converging when $r\rightarrow0.$\medskip
\end{proposition}

\begin{proposition}
\label{cpo} (Point $P_{0})$ 1) Assume that $q>m+1$ and $q+b<\frac{N-p}{p-1}%
\mu<N+b-m\frac{N-q}{q-1}.$ If $\gamma<\frac{N-p}{p-1}$ there exist
trajectories converging to $P_{0}$ when $r\rightarrow\infty$ (and not when
$r\rightarrow0).$ If $\gamma>\frac{N-p}{p-1}$ the same happens when
$r\rightarrow0$ (and not when $r\rightarrow\infty).$\medskip

2) Assume that $q<m+1$ and $q+b>\frac{N-p}{p-1}\mu>N+b-m\frac{N-q}{q-1}$ and
$q\frac{N-p}{p-1}\mu+m(N-q)\neq N(q-1)+(b+1)q.$ If $\gamma<\frac{N-p}{p-1}$
there exist trajectories converging to $P_{0}$ when $r\rightarrow0$ (and not
when $r\rightarrow\infty).$ If $\gamma>\frac{N-p}{p-1}$ there exist
trajectories converging when $r\rightarrow$ $\infty$ (and not when when
$r\rightarrow$ $0).$\medskip

In any case, setting $\kappa=\frac{1}{q-1-m}(\frac{N-p}{p-1}\mu-(q+b)),$ there
holds
\begin{equation}
\lim r^{\frac{N-p}{p-1}}u=\alpha>0,\quad\lim r^{\kappa}v=\beta>0. \label{aup}%
\end{equation}

\end{proposition}

\begin{remark}
This result improves the results of existence obtained by the fixed point
theorem in \cite{R} in the case of system $(RP)$ with $p=q=2,a=0,N=3,$
$2s+m\neq3.$ The proof is quite simpler..\medskip
\end{remark}

\begin{proposition}
\label{mis} (Point $I_{0})$ If $\frac{N-p}{p-1}s>N+a$ and $\frac{N-q}{q-1}%
\mu>N+b,$ there exist trajectories converging to $I_{0}$ when $r\rightarrow
\infty,$ and then
\begin{equation}
\lim_{r\rightarrow\infty}r^{\frac{N-p}{p-1}}u=\beta,\quad\lim_{r\rightarrow
\infty}v=\alpha>0. \label{lup}%
\end{equation}
For any $s,m\geqq0,$ there is no trajectory converging when $r\rightarrow$
$0.\medskip$
\end{proposition}

\begin{proposition}
\label{gzero} (Point $G_{0})$ Suppose $\frac{N-p}{p-1}\mu<N+b.$ If
$q+b<\frac{N-p}{p-1}\mu$ and $N+a<\frac{N-p}{p-1}s,$ there exist trajectories
converging to $G_{0}$ when $r\rightarrow$ $\infty.$ If $\frac{N-p}{p-1}%
\mu<q+b$ and $\frac{N-p}{p-1}s<N+a,$ the same happens when $r\rightarrow0$. In
any case
\begin{equation}
\lim r^{\frac{N-p}{p-1}}u=\beta,\quad\lim v=\alpha>0. \label{nup}%
\end{equation}
\medskip
\end{proposition}

\begin{proposition}
\label{mes} (Point $C_{0})$ Suppose $N+b<\frac{N-q}{q-1}m$ (hence $q<m+1)$
with $m\neq\frac{N(q-1)+(b+1)q}{N-q},$ and $\delta>\frac{(N+a)(m+1-q)}{q+b}$.
Then there exist trajectories converging to $C_{0}$ when $r\rightarrow$
$\infty$ (and not when $r\rightarrow$ $0),$ and then
\begin{equation}
\lim u=\alpha>0,\quad\lim r^{k}v=\beta, \label{rup}%
\end{equation}
where $k=\frac{q+b}{m+1-q}.$ $\medskip$
\end{proposition}

\begin{proposition}
\label{mas} (Point $R_{0})$ Assume that $N+b<\frac{N-q}{q-1}m$ (hence $q<m+1)$
with $m\neq\frac{N(q-1)+b+bq}{N-q},$ and $\delta<\frac{(N+a)(m+1-q)}{q+b}.$ If
$\frac{(p+a)(m+1-q)}{q+b}<\delta,$ there exist trajectories converging to
$R_{0}$ when $r\rightarrow\infty$ (and not when $r\rightarrow0)$. If
$\delta<\frac{(p+a)(m+1-q)}{q+b},$ there exist trajectories converging when
$r\rightarrow0$ (and not when $r\rightarrow\infty)$, and then (\ref{rup})
holds again.\medskip
\end{proposition}

We obtain similar results of convergence to the points $Q_{0},J_{0}%
,H_{0},D_{0},S_{0}$ by exchanging $p,\delta,s,a$ and $q,\mu,m,b$. There is no
admissible trajectory converginf to $0,K_{0},L_{0},$ see Remark \ref{none}.

\section{Global results for system $(S)$\label{glo}}

We are concerned by the existence of global positive solutions. First we find
again easily some known results by using our dynamical approach.

\begin{proposition}
\label{fac}Assume that system $(S)$ admits a positive solution $(u,v)$ in
$(0,\infty).$ Then the corresponding solution $(X,Y,Z,W)$ of system $(M)$
stays in the box%
\begin{equation}
\mathcal{A=}\left(  0,\frac{N-p}{p-1}\right)  \times\left(  0,\frac{N-q}%
{q-1}\right)  \times\left(  0,N+a\right)  \times\left(  0,N+b\right)  ,
\label{jom}%
\end{equation}
in other words
\begin{equation}
ru^{\prime}+\frac{N-p}{p-1}u>0,\quad rv^{\prime}+\frac{N-q}{q-1}v>0,\quad
r^{1+a}u^{s}v^{\delta}<(N+a)\left\vert u^{\prime}\right\vert ^{p-1},\quad
r^{1+b}u^{\mu}v^{m}<(N+b)\left\vert v^{\prime}\right\vert ^{q-1}. \label{mil}%
\end{equation}
and then
\begin{equation}
u^{s-p+1}v^{\delta}\leqq C_{1}r^{-(p+a)},\quad u^{\mu}v^{m-q+1}\leqq
C_{2}r^{-(q+b)},\quad\text{ in }(0,\infty), \label{sti}%
\end{equation}
where $C_{1}=(N+a)(\frac{N-p}{p-1})^{p-1},C_{2}=(N+b)(\frac{N-q}{q-1})^{q-1},$
and
\begin{equation}
\lim_{r\rightarrow0}r^{\frac{N-p}{p-1}}u=c_{1}\geqq0,\quad\lim_{r\rightarrow
0}r^{\frac{N-q}{q-1}}v(r)=c_{2}\geqq0,\quad\lim\inf_{r\rightarrow\infty
}r^{\frac{N-p}{p-1}}u>0,\quad\lim\inf_{r\rightarrow\infty}r^{\frac{N-q}{q-1}%
}v>0. \label{flic}%
\end{equation}
As a consequence if $s\leqq p-1$ or $m\leqq q-1,$ we have
\begin{equation}
u\leqq K_{1}r^{-\gamma},\quad v\leqq K_{2}r^{-\xi},\quad\text{in }(0,\infty),
\label{jim}%
\end{equation}
with $K_{1}=C_{1}^{(q-1-m)/D}C_{2}^{\delta/D},K_{2}=C_{1}^{\mu/D}%
C_{2}^{(p-1-s)/D}.$\medskip
\end{proposition}

\begin{proof}
The solution of system $(M)$ in $\mathcal{R}$ defined on $\mathbb{R}.$ On the
hyperplane $X=\frac{N-p}{p-1}$ we have $X_{t}>0,$ the field is going out. If
at some time $t_{0},$ $X(t_{0})=$ $\frac{N-p}{p-1},$ then $X(t)>\frac
{N-p}{p-1}$ for $t>t_{0},$ in turn $X_{t}\geqq X\left[  X-\frac{N-p}%
{p-1}\right]  >0,$ since since $Z>0,$ thus $X(t)>2\frac{N-p}{p-1}$ for
$t>t_{1}>t_{0},$ then $X_{t}\geqq X^{2}/2,$ which implies that $X$ blows up in
finite time; thus $X(t)<\frac{N-p}{p-1}$ on $\mathbb{R};$ in the same way
$Y(t)<\frac{N-q}{q-1}.$ On the hyperplane $Z=N+a$ we have $Z_{t}<0,$ the field
is entering. If at some time $t_{0},$ $Z(t_{0})=N+a$ then $Z(t)>N+a$ for
$t<t_{0},$ then $Z_{t}\leqq Z(N+a-Z),$ since $sX+\delta Y>0,$ and $Z$ blows up
in finite time as above; thus $Z(t)<N+a$ on $\mathbb{R},$ in the same way
$W(t)<N+b.$ Then (\ref{mil}),(\ref{sti}) and (\ref{jim}) follows. By
integration it implies that $r^{(N-p)/(p-1)}u(r)$ is nondecreasing near $0$ or
$\infty$, hence (\ref{flic}) holds.\medskip
\end{proof}

Next we prove Theorem \ref{equi}.\medskip

\begin{proof}
[Proof of Theorem \ref{equi}](i) The trajectories of the regular solutions
start from $N_{0}=(0,0,N+a,N+b)$, from Proposition \ref{ereg}, and the
unstable variety $\mathcal{V}_{u}$ has dimension 2, from (\ref{lib}),
(\ref{vib}). It is given locally by $Z=\varphi(X,Y),W=\psi(X,Y)$ for $(X,Y)\in
B(0,\rho)\backslash\left\{  0\right\}  \subset\mathbb{R}^{2}.\medskip$

To any $(x,y)\in B(0,\rho)\backslash\left\{  0\right\}  $ we associate the
unique trajectory $\mathcal{T}_{x,y}$ in $\mathcal{V}_{u}$ going through this
point. If $T^{\ast}$ is the maximal interval of existence of a solution on
$\mathcal{T}_{x,y}$, then $\lim_{t\rightarrow T^{\ast}}(X(t)+Y(t))=\infty.$
Indeed $Z,$ and $W$ satisfy $0<Z<N+a,$ $0<W<N+b$ as long as the solution
exists, because at a time $T$ where $Z(T)=N+a,$ we have $Z_{t}<0.$ If there
exists a first time $T$ such that $X(T)=\frac{N-p}{p-1}$ or $Y(T)=\frac
{N-q}{q-1},$ then $T<T^{\ast}.$ We consider the open rectangle $\mathcal{N}$
of submits
\[
(0,0),\quad\varpi_{1}=\left(  \frac{N-p}{p-1},0\right)  ,\quad\varpi
_{2}=\left(  0,\frac{N-q}{q-1}\right)  ,\quad\varpi=\left(  \frac{N-p}%
{p-1},\frac{N-q}{q-1}\right)  .
\]
Let $\mathcal{U}=\left\{  (x,y)\in B(0,\rho):x,y>0\right\}  $; then
$\mathcal{U=S}_{1}\cup\mathcal{S}_{2}\cup\mathcal{S}_{3}\cup\mathcal{S},$
where
\[
\left\{
\begin{array}
[c]{c}%
\mathcal{S}_{i}=\left\{  (x,y)\in\mathcal{U}:\mathcal{T}_{x,y}\text{ leaves
}\mathcal{N}\text{ on }\left(  \varpi_{i},\varpi\right)  \right\}  ,\quad
i=1,2,\\
\mathcal{S}_{3}=\left\{  (x,y)\in\mathcal{U}:\text{ }\mathcal{T}_{x,y}\text{
leaves }\mathcal{N}\text{ at }\varpi\right\}  ,\quad\mathcal{S}=\left\{
(x,y)\in\mathcal{U}:\text{ }\mathcal{T}_{x,y}\text{ stays in }\mathcal{N}%
\right\}  .
\end{array}
\right.
\]
Any element of $\mathcal{S}$ defines a G.S. Assume $s<\frac{N(p-1)+p+pa}%
{N-p}.$ Let us show that $\mathcal{S}_{1}$ is nonempty$.$ Consider the
trajectory $\mathcal{T}_{\bar{x},0}$ on $\mathcal{V}_{u}$ associated to
$(\bar{x},0),$ with $\bar{x}\in\left(  0,\rho\right)  ,$ going through
$\bar{M}=(\bar{x},0,$ $\varphi(\bar{x},0),\psi(\bar{x},0));$ it is
\underline{\textit{not admissible}} for our problem, since it is in the
hyperplane $Y=0$: it satisfies the system
\[
\left\{
\begin{array}
[c]{c}%
X_{t}=X\left[  X-\frac{N-p}{p-1}+\frac{Z}{p-1}\right]  ,\\
Z_{t}=Z\left[  N+a-sX-Z\right]  ,\\
W_{t}=W\left[  N+b-\mu X-W\right]  ,
\end{array}
\right.
\]
which is \underline{\textit{not completely coupled}}. The two equations in
$X,Z$ corresponds to the equation
\begin{equation}
-\Delta_{p}U=r^{a}U^{s}. \label{vic}%
\end{equation}
The regular solutions of (\ref{vic}) are changing sign, since $s$ is
subcritical, see Section \ref{scal}. Consider the solution $(\bar{X},\bar
{Y},\bar{Z},\bar{W})$ of system $(M),$ of trajectory $\mathcal{T}_{\bar{x},0}%
$, going through $\bar{M}$ at time $0;$ it satisfies $\bar{Y}=0,$ and $\bar
{X}(t)>0,$ $\bar{Z}(t)>0$ tend to $\infty$ in finite time $T^{\ast}$, then for
any given $C\geqq\frac{N-p}{p-1},$ there exist a first time $T<T^{\ast}$ such
that $\bar{X}(T)=C$, and $\bar{Y}(T)=0$. We have $\lim_{t\rightarrow-\infty
}\bar{W}=N+b,$ and necessarily $0<\bar{W}<N+b,$ in particular $0<\bar
{W}(T)<N+b;$ and $\bar{W}_{t}$ is bounded on $\left(  -\infty,T^{\ast}\right)
,$ then $\bar{W}$ has a finite limit at $T^{\ast}.\ $ The field at time $T$ is
transverse to the hyperplane $X=\frac{N-p}{p-1}$: we have $\bar{X}_{t}\geqq
C\frac{Z(T)}{p-1}>0,$ since $\bar{Z}(T)>0$. From the continuous dependance of
the initial data at time $0,$ for any $\varepsilon>0,$ there exists $\eta>0$
such that for any $(x,y)\in B((\bar{x},0),\eta)$ and for any $(X,Y,Z,W)$ on
$\mathcal{T}_{x,y},$ there exists a first time $T_{\varepsilon}$ such that
$X(T_{\varepsilon})=C,$ and $\left\vert Y(t)\right\vert \leqq\varepsilon$ for
any $t\leqq T_{\varepsilon}$, in particular for any $(x,y)$ $\in B((\bar
{x},0),\eta)$ with $y>0,$ and then $0<Y(t)\leqq\varepsilon$ for any $t\leqq
T_{\varepsilon}.$ Let us take $C=\frac{N-p}{p-1}.$ Then $(x,y)\in$
$\mathcal{S}_{1}$. The same arguments imply that $\mathcal{S}_{1}$ is open.
Similarly assuming $m<\frac{N(q-1)+q+qb}{N-q}$ implies that $\mathcal{S}_{2}$
is nonempty and open. By connexity $\mathcal{S}$ is empty if and only if
$\mathcal{S}_{3}$ is nonempty. \medskip

(ii) Here the difficulty is due to the fact that the zeros of $u,v$ correspond
to infinite limits for $X,Y,$ and then the argument of continuous dependance
is no more available. We can write $\mathcal{U=M}_{1}\cup\mathcal{M}_{2}%
\cup\mathcal{M}_{3}\cup\mathcal{S},$ where
\[
\left\{
\begin{array}
[c]{c}%
\mathcal{M}_{1}=\left\{  (x,y)\in\mathcal{U}\text{ and }\mathcal{T}%
_{x,y}\text{ has an infinite branch in }X\text{ with }Y\text{ bounded}%
\right\}  ,\\
\mathcal{M}_{2}=\left\{  (x,y)\in\mathcal{U}:\text{ }\mathcal{T}_{x,y}\text{
has an infinite branch in }Y\text{ with }X\text{ bounded}\right\}  ,\\
\mathcal{M}_{3}=\left\{  (x,y)\in\mathcal{U}:\text{ }\mathcal{T}_{x,y}\text{
has an infinite branch in }(X,Y)\right\}  .
\end{array}
\right.
\]
In other words, $\mathcal{M}_{1}$ is the set of $(x,y)\in\mathcal{U}$ such
that for any $(X,Y,Z,W)$ on $\mathcal{T}_{x,y},$ there exists a $T^{\ast}$
such that $\lim_{t\rightarrow T^{\ast}}X(t))=\infty,$ and $Y(t)$ stays bounded
on $\left(  -\infty,T^{\ast}\right)  ,$ that means the set of $(x,y)\in
\mathcal{U}$ such that for any solution $(u,v)$ corresponding to
$\mathcal{T}_{x,y},$ $u$ vanishes before $v;$ similarly for $\mathcal{M}_{2}.$
Otherwise $\mathcal{M}_{3}$ is the set of $(x,y)\in\mathcal{U}$ such that
there exists a $T^{\ast}$ such that $\lim_{t\rightarrow T^{\ast}}%
X(t)=\lim_{t\rightarrow T^{\ast}}Y(t)=\infty,$ that means $(u,v)$ vanish at
the same $R^{\ast}=e^{T^{\ast}}$. In that case, from the H\"{o}pf Lemma,
$\lim_{r\rightarrow R}\frac{u^{\prime}}{(r-R)u}=1,$ then $\lim_{t\rightarrow
T^{\ast}}\frac{X}{Y}=1.$

We are lead to show that $\mathcal{M}_{1}$ is nonempty and open for
$s<\frac{N(p-1)+p+pa}{N-p}$. We consider again the trajectory $\mathcal{\bar
{T}}$ and take $C$ large enough: $C=2(\frac{N-p}{p-1}+\frac{N+\left\vert
b\right\vert }{q-1}).$ Let $\varepsilon\in\left(  0,\frac{C}{2}\right)  .$ For
any $(x,y)\in B((\bar{x},0),\eta)$ with $y>0,$ and any $(X,Y,Z,W)$ on
$\mathcal{T}_{x,y},$ there is a first time $T_{\varepsilon}$ such that
$X(T_{\varepsilon})=C,$ and $0<Y(t)\leqq\varepsilon$ for any $t\leqq
T_{\varepsilon}.$ And $X$ is increasing and $X_{t}\geqq X(X-C),$ thus there
exists $T^{\ast}$ such that $\lim_{t\rightarrow T^{\ast}}X(t)=\infty.$ Setting
$\varphi=X/Y,$ we find
\[
\frac{\varphi_{t}}{\varphi}=X-Y+\frac{Z}{p-1}-\frac{W}{q-1}+\frac{N-q}%
{q-1}-\frac{N-p}{p-1}\geqq X-Y-\frac{C}{2}%
\]
then $\varphi_{t}(T_{\varepsilon})>0$. Let $\theta=\sup\left\{
t>T_{\varepsilon}:\varphi_{t}>0\right\}  ;$ suppose that $\theta$ is finite;
then $\varphi(\theta)>\varphi\left(  T_{\varepsilon}\right)  =C/\varepsilon>2$
and $X\left(  \theta\right)  \leqq Y\left(  \theta\right)  +C<X\left(
\theta\right)  /2+C,$ which is contradictory. Then $\varphi$ is increasing up
to $T^{\ast};$ if $\lim_{t\rightarrow T^{\ast}}Y(t)=\infty,$ then
$\lim_{t\rightarrow T^{\ast}}\varphi=1,$ which is impossible. Then
$(x,y)\in\mathcal{M}_{1},$ thus $\mathcal{M}_{1}$ is nonempty. In the same way
$\mathcal{M}_{1}$ is open. Indeed for any $(\bar{x},\bar{y})\in\mathcal{M}%
_{1}$ there exists $M>0$ such that $0<\bar{Y}(t)\leqq M/2$ on $\mathcal{T}%
_{\bar{x},\bar{y}}$. To conclude we argue as above, with $(\bar{x},0)$
replaced by $(\bar{x},\bar{y}),$ and $C$ replaced by $C+M.$\medskip
\end{proof}

\begin{proof}
[Proof of Proposition \ref{com}]Assume $s\geqq\frac{N(p-1)+p+pa}{N-p}.$
Consider the Pohozaev type function
\begin{equation}
\mathcal{F}(r)=r^{N}\left[  \frac{\left\vert u^{\prime}\right\vert ^{p}%
}{p^{\prime}}+\frac{r^{a}u^{s+1}}{s+1}v^{\delta}+\frac{N-p}{p}\frac
{u\left\vert u^{\prime}\right\vert ^{p-2}u^{\prime}}{r}\right]  =r^{N-p}%
u^{p}\left[  \frac{X}{p^{\prime}}+\frac{1}{s+1}Z-\frac{N-p}{p}\right]  .
\label{pla}%
\end{equation}
We find $\mathcal{F}(0)=0$ and
\begin{align}
\mathcal{F}^{\prime}(r)  &  =r^{N-1+a}\left[  \left(  \frac{N+a}{s+1}%
-\frac{N-p}{p}\right)  v^{\delta}u^{s+1}+\frac{\delta}{s+1}ru^{s+1}%
v^{\delta-1}v^{\prime}\right] \nonumber\\
&  =r^{N-1+a}v^{\delta}u^{s+1}\left[  \frac{N+a}{s+1}-\frac{N-p}{p}%
-\frac{\delta Y}{s+1}\right]  \label{plu}%
\end{align}
From our assumption, $\mathcal{F}$ is decreasing, and $Z>0,$ thus
$X<\frac{N-p}{p-1}.$ Then $\mathcal{S}_{1},\mathcal{S}_{3}$ are empty. If
moreover $m\geqq\frac{N(q-1)+q+qb}{N-q}$ then $\mathcal{S}_{2}$ is empty,
therefore $\mathcal{S}=\mathcal{U}.$\medskip
\end{proof}

\begin{remark}
Let us only assume that $s\geqq\frac{N(p-1)+p+pa}{N-p}.$ If one function has a
first zero, it is $v$. Indeed if there exists a first value $R$ where
$u(R)=0,$ and $v(r)>0$ on $\left[  0,R\right)  ,$ then $\mathcal{F}%
(R)=\frac{R^{N}}{p^{\prime}}\left\vert u^{\prime}(R)\right\vert ^{p}>0.$
\end{remark}

As a first consequence we obtain existence results for the Dirichlet problem.
It solves an open problem in the case $s>p-1$ or $m>q-1,$ and extends some
former results of \cite{CFMiT} and \cite{Z2}. Our proof, based on the shooting
method differs from the proof of \cite{CFMiT}, based on degree theory and
blow-up technique. Our results extend the ones of \cite[Theorem 2.2]{B1}
relative to the case $p=q=2,$ obtained by studying the equation satisfied by a
suitable function of $u,v.$

\begin{corollary}
\label{dir} system $(S)$ admits no G.S. and then there is a radial solution of
the Dirichlet problem in a ball in any of the following cases:\medskip

(i) $p<s+1,q<m+1,$\quad and\quad$\min(s\frac{N-p}{p-1}+\frac{N-q}{q-1}%
\delta-(N+a),\frac{N-p}{p-1}\mu+m\frac{N-q}{q-1}-(N+b))\leqq0;$\medskip

(ii) $p<s+1,$ $q>m+1$\quad and\quad$s\frac{N-p}{p-1}+\frac{N-q}{q-1}%
\delta-(N+a)\leqq0$ or $\gamma-\frac{N-p}{p-1}>0;$\medskip

(iii) $p>s+1,q>m+1$\quad and\quad$\max(\gamma-\frac{N-p}{p-1},\xi-\frac
{N-q}{q-1})\geqq0;$\medskip

(iv) $p\geqq s+1,q\geqq m+1$\quad and\quad$\max(\gamma-\frac{N-p}{p-1}%
,\xi-\frac{N-q}{q-1})>0.$\medskip
\end{corollary}

\begin{proof}
From Theorem \ref{equi}, we are reduced to prove the nonexistence of
G.S.\ \medskip

\noindent(i) Assume $p<s+1,$ and $s\frac{N-p}{p-1}+\frac{N-q}{q-1}%
\delta-(N+a)<0.$ We have $-\Delta_{p}u\geqq Cr^{a-\frac{N-q}{q-1}\delta}u^{s}$
for large $r.$ From \cite[Theorem 3.1]{BPo}, we find $u=O(r^{-(p+a-\frac
{N-q}{q-1}\delta)/(s+1-p)})$, and then $s\frac{N-p}{p-1}+\frac{N-q}{q-1}%
\delta-(N+a)\geqq0,$ from (\ref{flic}), which contradicts our assumption. In
case of equality, we find $-\Delta_{p}u\geqq Cr^{-N}$ for large $r,$ which is
impossible. Then there exists no G.S. This improves ythe result of
\cite{CFMiT} where the minimum is replaced by a maximum.\medskip

\noindent(ii) Assume $p<s+1,$ $q>m+1$ and $\gamma-\frac{N-p}{p-1}>0;$ then
$u=O(r^{-\gamma}),$ which contradicts (\ref{flic}). If $\gamma-\frac{N-p}%
{p-1}=0,$ then $\lim r^{\frac{N-p}{p-1}}u=\alpha>0,$ and $\xi>\frac{N-q}%
{q-1}.$ Hence $-\Delta_{q}v\geqq Cr^{b-\frac{N-p}{p-1}\mu}v^{m}$ for large
$r,$ then $v\geqq Cr^{(q+b-\frac{N-p}{p-1}\delta)/(q-1-m)}=Cr^{-\xi}.$ There
exists $C_{1}>0$ such that $C_{1}\leqq$ $r^{\xi}v\leqq2C_{1}$ for large $r,$
from \cite[Theorem 3.1]{BPo} and (\ref{jim}), then $-\Delta_{p}u\geqq Cr^{-N}$
for some $C>0,$ which is again contradictory.\medskip

\noindent(iii) (iv) The nonexistence of G.S is obtained by extension of the
proof of \cite{CFMiT} to the case $a,b\neq0.$ Moreover (iii) implies the
nonexistence of positive solution $(u,v)$, radial or not, in any exterior
domain $(R,\infty)\times(R,\infty),R>0$ from \cite{BPo}.\medskip
\end{proof}

\begin{corollary}
Assume (\ref{cd}) with $p=q=2.$ If $\delta+s\geqq\frac{N+2+2a}{N-2}$ and
$\mu+m\geqq\frac{N+2+2b}{N-2},$ then system $(S)$ admits a G.S.\medskip
\end{corollary}

\begin{proof}
It was shown in \cite{ReZ}, \cite{Z5} by the moving spheres method that the
Dirichlet problem has no radial or nonradial solution. Then Theorem \ref{equi}
applies again.\medskip
\end{proof}

We aso extend and improve a result of nonexistence of \cite{ChLuG} for the
case $p=q=2,a=0,s>1$:

\begin{proposition}
\label{pas}Assume $s+1>p$ or $\gamma>\frac{N-p}{p},$ and
\begin{equation}
s+\frac{p(N-q)}{(q-1)(N-p)}\delta<\frac{N(p-1)+pa+p}{N-p} \label{ouh}%
\end{equation}
Then system $(S)$ admits no $G.S.$ and then there is a solution of the
Dirichlet problem. The same happens by exchanging $p,s,\delta,a,\gamma$ with
$q,m,\mu,b,\xi.$\medskip
\end{proposition}

\begin{proof}
Consider the function $\mathcal{F}$ defined at (\ref{pla}). Suppose that there
exists a G.S. Then from (\ref{jom}) and (\ref{ouh}) we find
\[
\frac{N+a}{s+1}-\frac{N-p}{p}-\frac{\delta Y}{s+1}>\frac{N+a}{s+1}-\frac
{N-p}{p}-\frac{\delta}{s+1}\frac{N-q}{q-1}\geqq0.
\]
From (\ref{plu}), we deduce that $\mathcal{F}$ is nondecreasing. First suppose
$s+1>p.$ From (\ref{sti}) and (\ref{flic}),it follows that $u=O(r^{-k})$ at
$\infty,$ with $k=(p+a-\delta\frac{N-q}{q-1})/(s-p+1).$ In turn $r^{N-p}%
u^{p}=O(r^{(N-p)-kp})=o(1)$ from (\ref{ouh}), then $\mathcal{F}(r)=o(1)$ near
$\infty.$ Next assume $s+1\leqq p$ and $\gamma>\frac{N-p}{p}$. Then
$r^{N-p}u^{p}=O(r^{N-p-\gamma p}),$ hence $\mathcal{F}(r)=o(1)$ near $\infty$.
In any case we get a contradiction.
\end{proof}

\section{The Hamiltonian system\label{HS}}

Here we consider the nonnegative solutions of the variational Hamiltonian
problem $(SH)$ in $\Omega\subset\mathbb{R}^{N}$%
\[
(SH)\left\{
\begin{array}
[c]{c}%
-\Delta u=\left\vert x\right\vert ^{a}v^{\delta},\\
-\Delta v=\left\vert x\right\vert ^{b}u^{\mu},
\end{array}
\right.
\]
where $p=q=2<N,$ $s=m=0,$ $a>b>-2,$ and $D=\delta\mu-1>0.$ For this case we
find
\[
\gamma=\frac{(2+a)+(2+b)\delta}{D},\quad\xi=\frac{2+b+(2+a)\mu}{D},\quad
\gamma+2+a=\delta\xi,\quad\xi+2+b=\mu\gamma.
\]
The particular solution ($u_{0}(r),v_{0}(r))=(Ar^{-\gamma},Br^{-\xi})$ exists
for $0<\gamma<N-2,$ $0<\xi<N-2.$ Here $X,Y,Z,W$ are defined by%

\[
X(t)=\frac{r\left\vert u^{\prime}\right\vert }{u},\qquad Y(t)=\frac
{r\left\vert v^{\prime}\right\vert }{v},\qquad Z(t)=\frac{r^{1+a}v^{\delta}%
}{\left\vert u^{\prime}\right\vert },\qquad W(t)=\frac{r^{1+b}u^{\mu}%
}{\left\vert v^{\prime}\right\vert },
\]
with $t=\ln r,$ and system $(M)$ becomes%
\[
(MH)\left\{
\begin{array}
[c]{c}%
X_{t}=X\left[  X-(N-2)+Z\right]  ,\\
Y_{t}=Y\left[  Y-(N-2)+W\right]  ,\\
Z_{t}=Z\left[  N+a-\delta Y-Z\right]  ,\\
W_{t}=W\left[  N+b-\mu X-W\right]
\end{array}
\right.
\]
This system has a \underline{\textit{Pohozaev type}}\textbf{ }function, well
known at least in the case $a=b=0$, given at (\ref{Pom}):%
\begin{align*}
\mathcal{E}_{H}(r)  &  =r^{N}\left[  u^{\prime}v^{\prime}+r^{b}\frac
{\left\vert u\right\vert ^{\mu+1}}{\mu+1}+r^{a}\frac{\left\vert v\right\vert
^{\delta+1}}{\delta+1}+\frac{N+a}{\delta+1}\frac{vu^{\prime}}{r}+\frac
{N+b}{\mu+1}\frac{uv^{\prime}}{r}\right] \\
&  =r^{N-2}uv\left[  XY-\frac{Y(N+b-W)}{\mu+1}-\frac{(N+a-Z)X}{\delta
+1}\right] \\
&  =r^{N-2-\gamma-\xi}(ZX)^{(\mu+1)/D}(WY)^{(\delta+1)/D}\left[
XY-\frac{Y(N+b-W)}{\mu+1}-\frac{(N+a-Z)X}{\delta+1}\right]  .
\end{align*}
It can also be found by a direct computation, and $\mathcal{E}_{H}$ satisfies
\[
\mathcal{E}_{H}^{\prime}(r)=r^{N-1}u^{\prime}v^{\prime}\left(  \frac
{N+a}{\delta+1}+\frac{N+b}{\mu+1}-(N-2)\right)  .
\]
We define the \underline{\textit{critical case}}\textit{ }as the case where
$(\delta,\mu)$ lie on the hyperbola $\mathcal{H}_{0}$ given by
\begin{equation}
\frac{N+a}{\delta+1}+\frac{N+b}{\mu+1}=N-2,\text{ equivalently }\gamma
+\xi=N-2. \label{hyp}%
\end{equation}
In this case $\gamma=$ $\frac{N+b}{\mu+1},\xi=\frac{N+a}{\delta+1},$ and
$\mathcal{E}_{H}^{\prime}(r)\equiv0$. It corresponds to the existence of a
first integral of system $(M),$ which can also be expressed in the variables
\textrm{U}$=r^{\gamma}u,$\textrm{V}$=r^{\xi}v$ of Remark \ref{vauv}:%
\[
\mathcal{E}_{H}(r)=\mathrm{U}_{t}\mathrm{V}_{t}-\gamma\xi\mathrm{UV}%
\mathcal{+}\frac{\mathrm{U}^{\mu+1}}{\mu+1}+\frac{\mathrm{V}^{\delta+1}%
}{\delta+1}=C.
\]
The supercritical case is defined as the case where $(\delta,\mu)$ is above
$\mathcal{H},$ equivalently $\gamma+\xi<N-2$ and the subcritical case
corresponds to $(\delta,\mu)$ under $\mathcal{H}$.

\begin{remark}
The energy $\mathcal{E}_{H,0}$ of the particular solution associated to
$M_{0}$ is always negative, given by $\mathcal{E}_{H,0}=-\frac{D}{\left(
\mu+1\right)  (\delta+1)}r^{N-2-\gamma-\xi}X_{0}Y_{0}(Z_{0}X_{0})^{(\mu
+1)/D}(W_{0}Y_{0})^{(\delta+1)/D}.$\medskip
\end{remark}

\begin{remark}
In the case $a=b=0,$ it is known that there exists a solution of the Dirichlet
problem in any bounded regular domain $\Omega$ of $\mathbb{R}^{N},$ see for
example \cite{deFFe}, \cite{HuVo}; for general $a,b,$ some restrictions on the
coefficients appear, see \cite{LiuYa} and \cite{deFPeRo}.\medskip
\end{remark}

Next consider the critical and supercritical cases. When $a=b=0,$ there exists
no solution if $\Omega$ is starshaped, see \cite{VdV}. Here we show the
existence of G.S. for general $a,b$. The existence in the critical case with
$a=b=0$ was first obtained in \cite{L}, then in the supercritical case in
\cite{SZ2}, and uniqueness was proved in \cite{HuVo}, \cite{SZ2}. The proofs
of \cite{SZ2} are quite long due to regularity problems, when $\delta$ or
$\mu<1,$ which play no role in our quadratic system.

\begin{remark}
The particular case $\delta=\mu$ and $a=b$ is easy to treat. Indeed in that
case $u=v$ is a solution of the scalar equation $\Delta u+\left\vert
x\right\vert ^{a}\left\vert u\right\vert ^{\delta-1}u=0,$ for which the
critical case is given by $\delta=(N+2+2a)/(N-2).$ Moreover if system $(SH)$
admits a G.S., or a solution of the Dirichlet problem in a ball, it satisfies
$u=v,$ from \cite{B1}. Then we are completely reduced to the scalar case. In
particular, in the critical case, the G.S. are given explicitely by:
$u=v=c(K+r^{(2+a)})^{(2-N)/(2+a)},$ where $K=c^{\delta-1}/(N+a)(N-2);$ in
other words they satisfy (\ref{line}) with $X=Y$ and $Z=W,$ i.e.
\[
\frac{X(t)}{N-2}+\frac{Z(t)}{N+a}-1=0.
\]
Near $\infty,$ the G.S. is (obviously) symmetrical: it joins the points
$N_{0}$ and $A_{0}.$\medskip
\end{remark}

\begin{remark}
Consider the case $\delta=1,$ $a=b=0,$ which is the case of the biharmonic
equation%
\[
\Delta^{2}u=u^{\mu}.
\]
Recall that it is the only case where the conjecture (\ref{(C)}) was
completely proved by Lin in \cite{Lin}. In the critical case $\mu
=(N+4)/(N-4),$ the G.S. are also given explicitely, see \cite{HuVo}:
\[
u(r)=c(K+r^{2})^{(4-N)/2},\quad K=c^{\mu-1}/(N-4)(N-2)N(N+2).
\]
They satisfy the relation $XY=\frac{N-Z}{2}X+\frac{N-4}{2N}(N-W)Y,$ and
moreover we find that they are on an \underline{hyperplane}, of equation
\[
\frac{(N-2)X(t)}{N(N-4)}+\frac{Z(t)}{N}-1=0.
\]
Observe also that the G.S. is \underline{not symmetrical }\textbf{
}near\textbf{ }$\infty$: $u$ behaves like $r^{4-N}$ and $v$ behaves like
$r^{2-N}.$ The trajectory in the phase space joins the points $N_{0}$ and
$Q_{0}=(N-4,N-2,2,0).$\medskip
\end{remark}

\begin{proof}
[Proof of Theorem \ref{cla}]1) \textit{Existence or nonexistence
results:\medskip}

$\bullet$\textit{ }In the supercritical or critical case\textit{ } we apply
any of the two conditions of Theorem \ref{equi}: \textbf{ }Here\textbf{
}$\mathcal{E}_{H}(0)=0,$ and $\mathcal{E}_{H}$ is nonincreasing; there does
not exist solutions of $(M)$ such that at some time $T,$ $X(T)=Y(T)=N-2$,
because at the time $T,$
\[
XY-\frac{Y(N+b-W)}{\mu+1}-\frac{(N+a-Z)X}{\delta+1}=(N-2)\left[
N-2-\frac{N+a}{\delta+1}-\frac{N+b}{\mu+1}+\frac{W}{\mu+1}+\frac{Z}{\delta
+1}\right]  >0
\]
since $W>0,Z>0,$ thus $\mathcal{E}_{H}(e^{T})>0,$ which is impossible.
Otherwise there exists no solution of the Dirichlet problem in a ball
$B(0,R),$ because $\mathcal{E}_{H}(R)=R^{N}u^{\prime}(R)v^{\prime}(R)>0$ from
the H\"{o}pf Lemma. Then there exists a G.S.\ The uniqueness is proved in
\cite{HuVo}.\textit{\medskip}

$\bullet$ In the subcritical case there is no radial G.S.: it would satisfy
$\mathcal{E}_{H}(0)=0,$ and $\mathcal{E}_{H}$ is nondecreasing, $\mathcal{E}%
_{H}(r)\leqq Cr^{N-2-\gamma-\xi}$ from (\ref{jom}), and $\gamma+\xi>(N-2),$
then $\lim_{r\rightarrow\infty}\mathcal{E}_{H}(r)=0$. From Theorem \ref{equi},
there exists a solution of the Dirichlet problem.\medskip\ 

2) \textit{Behaviour of the }$G.S.$\textit{ in the critical case. }

It is easy to see that the condition (\ref{ero}) implies $\mu>\frac{2+b}{N-2}$
and $\delta>\frac{2+a}{N-2},$ and that $\delta\leqq\frac{N+a}{N-2}$ and
$\mu\leqq\frac{N+b}{N-2}$ cannot hold simultaneously. One can suppose that
$\delta>\frac{N+a}{N-2}.$ Let $\mathcal{T}$ be the unique trajectory of the
G.S.. Then $\mathcal{E}_{H}(0)=0,$ thus $\mathcal{T}$ lies on the variety
$\mathcal{V}$ of energy $0$, defined by
\begin{equation}
\frac{X(N+a-Z)}{\delta+1}+\frac{Y(N+b-W)}{\mu+1}=XY. \label{var}%
\end{equation}
From (\ref{mil}) $\mathcal{T}$ starts from the point $N_{0},$ and from
(\ref{jom}) $\mathcal{T}$ stays in
\[
\mathcal{A=}\left\{  (X,Y,Z,W)\in\mathbb{R}^{4}:0<X<N-2,\quad0<Y<N-2,\quad
0<Z<N+a,\quad0<W<N+b\right\}  .
\]
\medskip(i) Suppose that $\mathcal{T}$ converges to a fixed point of the
system in $\mathcal{\bar{R}}$. Then the only possible points are $A_{0}%
,P_{0},Q_{0}$ which are effectively on $\mathcal{V}$. Indeed $I_{0},$
$J_{0},G_{0},H_{0}\not \in \mathcal{V}.$ But $Q_{0}=((N-2)\delta
-(2+a),N-2,N+a-(N-2)\delta,0)\not \in \mathcal{\bar{R}}$, since $\delta
>\frac{N+a}{N-2}$. And $P_{0}\in\mathcal{\bar{R}}$ if and only if $\mu
\leqq\frac{N+b}{N-2}$. \medskip

If $\mu>\frac{N+b}{N-2},$ then $\mathcal{T}$ converges to $A_{0}$. If
$\mu<\frac{N+b}{N-2},$ no trajectory converges to $A_{0},$ from Proposition
\ref{cao}, thus $\mathcal{T}$ converges to $P_{0}$. If $\mu\neq\frac{N+b}%
{N-2}$ the convergence is exponential, thus the behaviour of $u,v$ follows. If
$\mu=\frac{N+b}{N-2},$ then $\mathcal{T}$ converges converges to $A_{0,}%
=P_{0};$ the eigenvalues given by (\ref{fgh}) satisfy $\lambda_{1}=\lambda
_{2}=N-2,$ $\lambda_{3}=N+a-\delta(N-2)<0$ and $\lambda_{4}=0;$ the projection
of the trajectory on the hyperplane $Y=N-2$ satisfies the system
\[
X_{t}=X\left[  X-(N-2)+Z\right]  ,\qquad Z_{t}=Z\left[  N+a-\delta
(N-2)-Z\right]
\]
which presents a saddle point at $(N-2,0)$, thus the convergence of $X$ and
$Z$ is exponential, in particular we deduce the behaviour of $u.$ The
trajectory enters by the central variety of dimension $1,$ and by computation
we deduce that $Y-(N-2)=-t^{-1}+O(t^{-2+\varepsilon})$ near $\infty,$ and the
behaviour of $v$ follows.\medskip

\noindent(ii) Let us show that $\mathcal{T}$ converges to a fixed point. We
eliminate $W$ from (\ref{var}) and we get a still quadratic system in
$(X,Y,Z):$%
\begin{equation}
\left\{
\begin{array}
[c]{c}%
X_{t}=X\left[  X-(N-2)+Z\right]  ,\\
Y_{t}=Y\left[  Y+b+2-(\mu+1)X\right]  +\frac{\mu+1}{\delta+1}X(N+a-Z),\\
Z_{t}=Z\left[  N+a-\delta Y-Z\right]  .
\end{array}
\right.  \label{tro}%
\end{equation}
We have $X_{t}\geqq0,$ and $Y_{t}\geqq0$ near $-\infty.$ Suppose that $X$ has
a maximum at $t_{0}$ followed by a minimum at $t_{1}.$ At these times
$X_{tt}=XZ_{t}$ , thus we find $Z_{t}(t_{0})<0<Z_{t}(t_{1}).$ There exists
$t_{2}\in\left(  t_{0},t_{1}\right)  $ such that $Z_{t}(t_{2})=0,$ and $t_{2}$
is a minimum. At this time $Z(t_{2})=N+a-\delta Y(t_{2}),$ $Z_{tt}%
(t_{2})=-\delta(ZY_{t})(t_{2})$ hence
\[
Y_{t}(t_{2})=Y(t_{2})\left[  Y(t_{2})+b+2-\frac{\mu+1}{\delta+1}%
X(t_{2})\right]  <0
\]
and $X_{t}(t_{2})<0,$ hence ($X+Z)(t_{2})<N-2,$ and
\[
N-2-X(t_{2})>Z(t_{2})>N+a-\delta(\frac{\mu+1}{\delta+1}X(t_{2})-b-2)
\]%
\[
(a+2)+\delta(b+2)<(\delta\frac{\mu+1}{\delta+1}-1)X(t_{2})=\frac
{\delta(2+b)+(2+a)}{(N-2)\delta-(2+a)}X(t_{2})
\]
but $X(t_{2})<X(t_{0})<\delta(N-2)-(2+a),$ which is contradictory. Then $X$
has at most one extremum, which is a maximum, and then it has a limit in
$\left(  0,N-2\right]  $ at $\infty.$ In the same way, by symmetry, $Y$ has at
most one extremum, which is a maximum, and has a limit in $\left(
0,N-2\right]  $ at $\infty.$ Then $Z$ has at most one extremum, which is a
minimum. Indeed at the points where $Z_{t}=0,$ $-Z_{tt}$ has the sign of
$Y_{t}$. Thus $Z$ has a limit in $\left[  0,N+a\right)  $, similarly $W$ has a
limit in $\left[  0,N+b\right)  .\medskip$
\end{proof}

\textbf{Open problems: }1) For the case\textbf{ }$\delta=\mu,$ in the critical
case it is well known that there exist solutions $(u,v)$ of system $(SH)$ of
the form $(u,u),$ such that $r^{\gamma}u$ is periodic in $t=\ln r.$ They
correspond to a periodic trajectory for the scalar system $(M_{scal})$ with
$p=2,$ and it admits an infinity of such trajectories. If $\delta\neq\mu,$
does there exist solutions $(u,v)$ such that $(r^{\gamma}u,r^{\xi}v)$ is
periodic in $t,$ in other words a periodic trajectory for system
$(MH)?\medskip$

2)\textbf{ }In the\textbf{ }supercritical case, we cannot prove that the
regular trajectory $\mathcal{T}$ converges to $M_{0},$ that means
$\lim_{r\rightarrow\infty}r^{\gamma}u=A,$ $\lim_{r\rightarrow\infty}r^{\xi
}v=B.$ Here $\mathcal{E}_{H}(0)=0,$ $\mathcal{E}_{H}$ is nonincreasing, then
$\mathcal{E}_{H}$ is negative. The only fixed points of negative energy are
$M_{0},$ $G_{0},H_{0},$ but a G.S. satisfies (\ref{jim}), then it tends to
$(0,0)$ at $\infty,$ hence $\mathcal{T}$ cannot converge to $G_{0}$ or $H_{0}$
from Proposition \ref{gzero}; but we cannot prove that $\mathcal{T}$ converges
to some fixed point.

\section{A nonvariational system\label{NV}}

Here we consider system $(S)$ with $p=q=2,a=b$ and $s=m\neq0.$
\[
(SN)\left\{
\begin{array}
[c]{c}%
-\Delta u=\left\vert x\right\vert ^{a}u^{s}v^{\delta},\\
-\Delta v=\left\vert x\right\vert ^{a}u^{\mu}v^{s},
\end{array}
\right.
\]
where $D=\delta\mu-(1-s)^{2}$ $>0.$ In order to prove Theorem we can reduce
the system to the case $a=0,$ by changing $N$ into $\hat{N}=\frac{2(N+a)}%
{2+a},$ from Remark \ref{duc}; thus we assume $a=0$ in this Section. Here
\[
X=-\frac{ru^{\prime}}{u}\qquad Y=-\frac{rv^{\prime}}{v},\qquad Z(t)=-\frac
{ru^{s}v^{\delta}}{u^{\prime}},\qquad W(t)=-\frac{ru^{\mu}v^{s}}{v^{\prime}},
\]
and system $(M)$ becomes
\[
(MN)\left\{
\begin{array}
[c]{c}%
X_{t}=X\left[  X-(N-2)+Z\right]  ,\\
Y_{t}=Y\left[  Y-(N-2)+W\right]  ,\\
Z_{t}=Z\left[  N-sX-\delta Y-Z\right]  ,\\
W_{t}=W\left[  N-\mu X-sY-W\right]  .
\end{array}
\right.
\]
We have chosen this system because it is not variational, and different
hyperbolas in the plane $(\delta,\mu)$: \medskip

$\bullet$ the hyperbola $\mathcal{H}_{s}$ for which the linearized system at
$M_{0}$ has two imaginary roots, given by
\[
(\mathcal{H}_{s})\qquad\frac{1}{\delta+1-s}+\frac{1}{\mu+1-s}=\frac
{N-2}{N-(N-2)s}%
\]
whenever $s<\frac{N}{N-2},$ and $\delta+1-s>0,$ $\mu+1-s>0,$ from Proposition
\ref{cent};\medskip

$\bullet$ the hyperbola $\mathcal{H}_{0}$ defined by
\begin{equation}
(\mathcal{H}_{0})\qquad\frac{1}{\delta+1}+\frac{1}{\mu+1}=\frac{N-2}{N};
\label{ach}%
\end{equation}
it was shown in \cite{Mi} that above $\mathcal{H}_{0}$ there exists no
solution of the Dirichlet problem;\medskip

$\bullet$ an hyperbola $\mathcal{Z}_{s}$ introduced in \cite{Zh} in case
$s<\frac{N}{N-2},$ and $\min(\delta,\mu)>\left\vert s-1\right\vert :$
\begin{equation}
(\mathcal{Z}_{s})\qquad\frac{1}{\delta+1}+\frac{1}{\mu+1}=\frac{N-2}%
{N-(N-2)s}, \label{ine}%
\end{equation}

$\bullet$ we introduce the new curve $\mathcal{C}_{s}$ defined \underline
{\textit{for any }$s>0$} by
\[
(\mathcal{C}_{s})\qquad\text{ }\frac{N}{\mu+1}+\frac{N}{\delta+1}%
=N-2+\frac{(N-2)s}{2}\min(\frac{1}{\mu+1},\frac{1}{\delta+1}),
\]

We first extend and complete the results of \cite{Zh} and \cite{Mi}:

\begin{proposition}
\label{mac}(i) Assume $s<\frac{N}{N-2},$ and $\delta+1-s>0,$ $\mu+1-s>0.$
Under the hyperbola $\mathcal{Z}_{s},$ system $(SN)$ admits no G.S., and then
there is a solution of the Dirichlet problem in a ball.\medskip

(ii) Above $\mathcal{H}_{0}$ there exists no solution of the Dirichlet
problem. Thus there exists a G.S.\medskip
\end{proposition}

\begin{proof}
(i) We consider an energy function with parameters $\alpha,\beta,\sigma
,\theta:$
\begin{align}
\mathcal{E}_{N}(r)  &  =r^{N}\left[  u^{\prime}v^{\prime}+\alpha u^{\mu
+1}v^{s}+\beta v^{\delta+1}u^{s}+\frac{\sigma}{r}vu^{\prime}+\frac{\theta}%
{r}uv^{\prime}\right] \label{bbb}\\
&  =r^{N-2}uv\Psi_{0}=r^{N-2-\gamma-\xi}(ZX)^{\xi/2}(WZ)^{\gamma/2}\Psi_{0},
\end{align}
from (\ref{for}), where
\begin{equation}
\text{ }\Psi_{0}(X,Y,Z,W)=XY+\alpha WY+\beta ZX-\sigma X-\theta Y. \label{psi}%
\end{equation}
We get
\begin{align*}
r^{1-N}(uv)^{-1}\mathcal{E}_{N}^{\prime}(r)  &  =(\sigma+\theta
-(N-2))XY+(N\alpha-\theta)YW+(N\beta-\sigma)XZ\\
&  -(\alpha(\mu+1)-1)XYW-(\beta(\delta+1)-1)XYZ-\alpha sY^{2}W-\beta sX^{2}Z.
\end{align*}
Taking $\alpha=\frac{1}{\mu+1},\beta=\frac{1}{\delta+1},$ we find
\begin{equation}
r^{3-N}(uv)^{-1}\mathcal{E}_{N}^{\prime}(r)=(\sigma+\theta-(N-2))XY+(N\alpha
-\theta-\alpha sY)YW+(N\beta-\sigma-\beta sX)XZ. \label{stu}%
\end{equation}
If there exists a G.S., from (\ref{jom}) it satisfies $X,Y<N-2,$ hence
\begin{equation}
r^{3-N}(uv)^{-1}\mathcal{E}_{N}^{\prime}(r)>(\sigma+\theta
-(N-2))XY+((N-(N-2)s)\alpha-\theta)YW+((N-(N-2)s)\beta-\sigma)XZ. \label{sto}%
\end{equation}
Taking $\theta=\frac{N-(N-2)s}{\mu+1},\sigma=\frac{N-(N-2)s}{\delta+1},$ we
deduce that $\mathcal{E}_{N}^{\prime}>0$ under $\mathcal{Z}_{s}$. Moreover
$\mathcal{Z}_{s}$ is under $\mathcal{H}_{s},$ thus $\gamma+\xi>N-2.$ Then
$\mathcal{E}_{N}(r)=O(r^{N-2-\gamma-\xi})$ tends to $0$ at $\infty,$ which is
contradictory.\medskip

\noindent(ii) Taking $\alpha=\frac{1}{\mu+1}=\frac{\theta}{N},\beta=\frac
{1}{\delta+1}=\frac{\sigma}{N},$ it comes from (\ref{stu})
\[
r^{3-N}(uv)^{-1}\mathcal{E}_{N}^{\prime}(r)=(\frac{N}{\delta+1}+\frac{N}%
{\mu+1}-(N-2))XY-\alpha sY^{2}W-\beta sX^{2}Z
\]
hence $\mathcal{E}_{N}^{\prime}<0$ when (\ref{ach}) holds. At the value $\ R$
where $u(R)=v(R)=0,$ we find $\mathcal{E}_{N}(R)=R^{N}u^{\prime}(R)v^{\prime
}(R)>0,$ which is a contradiction.
\end{proof}

\begin{remark}
(i)When the four curves are simultaneously defined, they are in the following
order, from below to above: $\mathcal{Z}_{s},\mathcal{H}_{s},\mathcal{C}%
_{s},\mathcal{H}_{0}.$ They intersect the diagonal $\delta=\mu$ repectively
for
\[
\delta=\frac{N+2}{N-2}-2s,\quad\delta=\frac{N+2}{N-2}-s,\quad\delta=\frac
{N+2}{N-2}-\frac{s}{2},\quad\delta=\frac{N+2}{N-2}.
\]

(ii) For $\delta=\mu,$ system $(SN)$ has a G.S. for $\delta\geqq\frac
{N+2}{N-2}-s.$ Indeed it admits solutions of the form $(U,U)$, where $U$ is a
solution of equation $-\Delta U=U^{s+\delta}.$ Suppose moreover $s\leqq
\delta.$ If $1-s<\delta<\frac{N+2}{N-2}-s,$ then there exists no G.S; indeed
all such solutions satisfy $u=v,$ from \cite[Remark 3.3]{B1}. Then the point
$P_{s}=\left(  \frac{N+2}{N-2}-s,\frac{N+2}{N-2}-s\right)  $ appears to be the
separation point on the diagonal; notice that $P_{s}\in\mathcal{H}_{s}%
.$\bigskip
\end{remark}

Next we prove our main existence result of existence of a G.S. valid without
restrictions on $s$. The main idea is to introduce a \underline{\textit{new
energy function }$\Phi$ \textit{by adding two terms in }$X^{2}$\textit{ and
}$Y^{2}$} to the energy $\mathcal{E}_{N}$ defined at (\ref{bbb}). It is
constructed in order that $\Phi^{\prime}$ does not contain $Y$ and $Z.$ Then
we consider the set of couples $(X,Y)$ such that $\Phi^{\prime}$ has a sign,
which is bounded by a cubic curve. When ($\delta,\mu)$ is above $\mathcal{C}%
_{s}$, the cubic curve is exterior to the square
\begin{equation}
K=\left[  0,N-2\right]  \times\left[  0,N-2\right]  , \label{ki}%
\end{equation}
and then we can apply Theorem \ref{equi}.\medskip

\begin{proof}
[Proof of Theorem \ref{cru}]From Theorem \ref{equi}, if $s\geqq\frac{N+2}%
{N-2},$ all the regular solutions are G.S.. Thus we can assume $s<\frac
{N+2}{N-2}.$ Let $j,k\in\mathbb{R}$ be parameters, and
\begin{align*}
\Phi(r)  &  =\mathcal{E}_{N}(r)+r^{N}\left[  k\frac{s}{2}\frac{vu^{\prime2}%
}{u}+j\frac{s}{2}\frac{uv^{\prime2}}{v}\right] \\
&  =r^{N}\left[  u^{\prime}v^{\prime}+\alpha u^{\mu+1}v^{s}+\beta v^{\delta
+1}u^{s}+\frac{\sigma}{r}vu^{\prime}+\frac{\theta}{r}uv^{\prime}+k\frac{s}%
{2}\frac{vu^{\prime2}}{u}+j\frac{s}{2}\frac{uv^{\prime2}}{v}\right] \\
&  =r^{N-2}uv\Psi=r^{N-2-\gamma-\xi}(ZX)^{\xi/2}(WY)^{\gamma/2}\Psi,
\end{align*}
where
\[
\Psi(X,Y,Z,W)=XY+\alpha WY+\beta ZX-\sigma X-\theta Y+k\frac{s}{2}X^{2}%
+j\frac{s}{2}Y^{2}.
\]
Then
\begin{align*}
r^{3-N}(uv)^{-1}\Phi^{\prime}(r)  &  =(\sigma+\theta-(N-2))XY+(N\alpha
-\theta)YW+(N\beta-\sigma)XZ\\
&  -(\alpha(\mu+1)-1)XYW-(\beta(\delta+1)-1)XYZ+(j-\alpha)sY^{2}%
W+(k-\beta)sX^{2}Z\\
&  +ksX^{2}\left[  X-(N-2)\right]  +jsY^{2}\left[  Y-(N-2)\right]
+(N-2-X-Y)(k\frac{s}{2}X^{2}+j\frac{s}{2}Y^{2}).
\end{align*}
We eliminate the terms in $Z,W$ by taking $j=\alpha=\frac{1}{\mu+1},$
$k=\beta=\frac{1}{\delta+1},$ $\theta=N\alpha,$ $\sigma=N\beta.$ Then we get
the function $\Phi$ defined at (\ref{Pon}). Computing its derivative, we
obtain after reduction%
\begin{align*}
\mathcal{B(}X,Y)  &  :=-\frac{2}{s}r^{3-N}(uv)^{-1}\Phi^{\prime}(r)\\
&  =\beta X^{2}(N-2-X)+\alpha Y^{2}(N-2-Y)+XY\left[  \beta X+\alpha Y+\frac
{2}{s}(N-2-N\alpha-N\beta)\right]  .
\end{align*}

From Proposition \ref{mac} we can assume that $N(\alpha+\beta)-(N-2)>0$. We
determine the sign of $\mathcal{B}$ on the boundary $\partial K$ of the square
$K$ defined at (\ref{ki}). We have $\mathcal{B}(0,Y)=\alpha Y^{2}%
(N-2-Y)\geqq0$ and $\mathcal{B}(X,0)=\beta X^{2}(N-2-X)\geqq0$. In particular
$\mathcal{B}(0,0)=0.$ Otherwise $\mathcal{B}(N-2,Y)=Y\Theta(Y)$ with
\[
\Theta(Y)=\alpha Y\left[  2(N-2)-Y\right]  +(N-2)((N-2)\beta+\frac{2}%
{s}(N-2-N\alpha-N\beta).
\]
On the interval $\left[  0,N-2\right]  ,$ there holds $\Theta(Y)>\Theta(0)$.
By hypothesis, $(\delta,\mu)$ is above $\mathcal{C}_{s},$ or equivalently
\begin{equation}
(\alpha+\beta)\frac{N}{N-2}-1\leqq\frac{s}{2}\min(\alpha,\beta); \label{cdm}%
\end{equation}
consequently $\mathcal{B}(N-2,Y)\geqq0$ and similarly $\mathcal{B}%
(X,N-2)\geqq0.$ Then $\mathcal{B}$ is nonnegative on $\partial K$ and is zero
at $(0,0),(0,N-2),(N-2,0)$. The curve $\mathcal{B(}X,Y)=0$ is a cubic with a
double point at $(0,0),$ which is isolated under the condition (\ref{cdm}):
$\mathcal{B(}X,Y)>0$ near $(0,0),$ except at this point. Then $\mathcal{B(}%
X,Y)>0$ on the interior of $K.\medskip$

Suppose that there exists a regular solution such that $X(T)=Y(T)=N-2$ at the
same time $T.$ Indeed up to this time $(X,Y)$ stays in $K$, thus the function
$\Phi$ is decreasing. We have $\Phi(0)=0,$ and at the value $R=e^{T},$ we
find
\[
\Phi(R)=R^{N-2-\gamma-\xi}(N-2)^{\xi+\gamma+2}\left[  \frac{\alpha W+\beta
Z}{N-2}+1-(\beta+\alpha)(\frac{N}{N-2}-\frac{s}{2})\right]
\]
then $\Phi(R)>0$, since $\min(\alpha,\beta)<\alpha+\beta$. Therefore from
Theorem \ref{equi}, there exists a G.S.\medskip
\end{proof}

\begin{remark}
We wonder if the limit curve for existence of G.S. would be $\mathcal{H}_{s}$,
or another curve $\mathcal{L}_{s}$ defined by
\[
(\mathcal{L}_{s})\qquad\frac{1}{\delta+1}+\frac{1}{\mu+1}=\frac{N-2}%
{N-\frac{(N-2)s}{2}},
\]
which ensures that $\Phi(R)>0,$ and also $\mathcal{B}(N-2,N-2)>0.$ This curve
cuts the diagonal at the same point $P_{s}$ $=\left(  \frac{N+2}{N-2}%
-s,\frac{N+2}{N-2}-s\right)  $ as $\mathcal{H}_{s}.$ Notice that
$\mathcal{L}_{s}$ is under $\mathcal{H}_{s}.$
\end{remark}

\section{The radial potential system\label{RP}}

Here we study the nonnegative radial solutions of system $(SP):$%
\[
(SP)\left\{
\begin{array}
[c]{c}%
-\Delta_{p}u=\left\vert x\right\vert ^{a}u^{s}v^{m+1},\\
-\Delta_{q}v=\left\vert x\right\vert ^{a}u^{s+1}v^{m},
\end{array}
\right.
\]
with $a=b,\delta=m+1,\mu=s+1,$ and we assume (\ref{ht}). System $(M)$ becomes%
\[
(MP)\left\{
\begin{array}
[c]{c}%
X_{t}=X\left[  X-\frac{N-p}{p-1}+\frac{Z}{p-1}\right]  ,\\
Y_{t}=Y\left[  Y-\frac{N-q}{q-1}+\frac{W}{q-1}\right]  ,\\
Z_{t}=Z\left[  N+a-sX-(m+1)Y-Z\right]  ,\\
W_{t}=W\left[  N+b-(s+1)X-mY-W\right]  .
\end{array}
\right.
\]
For this system $D,$ $\gamma$ and $\xi$ are defined by%
\[
D=p(1+m)+q(1+s)-pq,\qquad(p-1-s)\gamma+p+a=(m+1)\xi,\qquad(q-1-m)\xi
+q+b=(s+1)\gamma,
\]
thus $\gamma$ and $\xi$ are linked independtly of $s,m$ by the relation
\begin{equation}
p(\gamma+1)=q(\xi+1)=\frac{pq(m+s+2+a)}{D}. \label{sli}%
\end{equation}
The system \thinspace is \textit{variational}. It admits an energy function,
given at (\ref{ppo}), which can also can be obtained by a direct computation
in terms of $X,Y,Z,W$:%

\begin{equation}
\mathcal{E}_{P}(r)=\psi\left[  ZW-\frac{s+1}{p}W((N-p)-(p-1)X)-\frac{m+1}%
{q}Z((N-q)-(q-1)Y)\right]  , \label{ene}%
\end{equation}
where
\[
\psi=\frac{r^{N-2-a}\left\vert u^{\prime}\right\vert ^{p-1}\left\vert
v^{\prime}\right\vert ^{q-1}}{u^{s}v^{m}}=r^{N-(\gamma+1)p}\left[
X^{q(s+1)(p-1)}Y^{p(m+1)(q-1)}Z^{p(q-m-1)}W^{q(p-s-1)}\right]  ^{1/D}.
\]
Then we find
\[
\mathcal{E}_{P}^{\prime}(r)=(N+a-(s+1)\frac{N-p}{p}-(m+1)\frac{N-q}%
{q})r^{N-1+a}u^{s+1}v^{m+1}.
\]
Thus we define a critical line $\mathcal{D}$ as the set of $\left(  \delta
,\mu\right)  =(m+1,s+1)$ such that
\begin{equation}
N+a=(m+1)\frac{N-q}{q}+\left(  s+1\right)  \frac{N-p}{p}, \label{cho}%
\end{equation}
equivalent to $pq(m+s+2+a)=ND,$ or $N+a=(m+1)\xi+(s+1)\gamma,$ or
\[
(\gamma,\xi)=(\frac{N-p}{p},\frac{N-q}{q})
\]
The subcritical case is given by the set of points under $\mathcal{D}$,
equivalently $\gamma>\frac{N-p}{p}$, $\xi>\frac{N-q}{q}$ or $\left(
s+1\right)  \gamma+(m+1)\xi>N+a.$ The supercritical case is the set of points
above $\mathcal{D}.$\medskip

\begin{remark}
The energy $\mathcal{(\mathcal{E}_{P})}_{0}$ of the particular solution
associated to $M_{0}$ is still negative$:$ $(\mathcal{E}_{P}\mathcal{)}%
_{0}=-\frac{D}{pq}r^{N+a-(\gamma+1)p}\left[  X_{0}^{q(p-1)}Y_{0}^{p(q-1)}%
Z_{0}^{q(s+1)}W_{0}^{p(m+1)}\right]  ^{1/D}.$\medskip
\end{remark}

\begin{remark}
\label{bah}When $p=q=2$, another energy function can be associated to the
transformation given at Remark \ref{vauv}: the system (\ref{uvs}) relative to
$u(r)=r^{-\gamma}$\textrm{U}$(t),\quad v(r)=r^{-\xi}$\textrm{V}$(t)$ is
\begin{equation}
\left\{
\begin{array}
[c]{c}%
\mathrm{U}_{tt}+(N-2-2\gamma)\mathrm{U}_{t}-\gamma(N-2-\gamma)\mathrm{U}%
+\mathrm{U}^{s}\mathrm{V}^{m+1}=0\\
\mathrm{V}_{tt}+(N-2-2\gamma)\mathrm{V}_{t}-\gamma(N-2-\gamma)\mathrm{V}%
+\mathrm{U}^{s+1}\mathrm{V}^{m}=0
\end{array}
\right.  \label{hhh}%
\end{equation}
and the function
\begin{equation}
E_{P}(t)=\frac{s+1}{2}(\mathrm{U}_{t}^{2}-\gamma(N-2-\gamma)\mathrm{U}%
^{2})+\frac{m+1}{2}(\mathrm{V}_{t}^{2}-\gamma(N-2-\gamma)\mathrm{V}%
^{2}+\mathrm{U}^{s+1\mathrm{V}m+1} \label{vvv}%
\end{equation}
satisfies%
\[
(E_{P})_{t}=-(N-2-2\gamma)\left[  (s+1)\mathrm{U}_{t}^{2}+(m+1)\mathrm{V}%
_{t}^{2}\right]
\]
It differs from $\mathcal{E}_{P},$ even in the critical case. This point is
crucial for Section \ref{nonradial}.\medskip
\end{remark}

It has been proved in \cite{TV}, \cite{TV2}, that in the subcritical case with
$a=0$, there exists a solution of the Dirichlet problem in any bounded regular
domain $\Omega$ of $\mathbb{R}^{N};$ and in the supercritical case there
exists no solution if $\Omega$ is starshaped. Here we prove two results of
existence or nonexistence of G.S. which seem to be new:\medskip

\begin{proof}
[Proof of Theorem \ref{clo}]1) \textit{Existence or nonexistence
results.}\medskip

$\bullet$\textit{ }In the supercritical or critical case there exists a G.S.
From Theorem \ref{equi}, if it were not true, then \textbf{ }there would exist
regular positive solutions of $(MP)$ such that $X(T)=\frac{N-p}{p-1}$ and
$Y(T)=\frac{N-q}{q-1}$. It would satify $\mathcal{E}_{P}\leqq0.$ Then at time
$T,$ we find $\mathcal{E}_{P}(R)>0,$ from (\ref{ene}), since $W>0,Z>0,$ which
is impossible.\medskip

$\bullet$\textit{ }In the subcritical case, there exists no G.S. Suppose that
there exists one. Now $\mathcal{E}_{P}$ is nondecreasing, hence $\mathcal{E}%
_{P}\geqq0.$ Its trajectory stays in the box $\mathcal{A}$ defined by
(\ref{jom}), thus it is bounded. If $q\geqq m+1$ and $p\geqq s+1,$ we deduce
that , $\mathcal{E}_{P}(r)=O(r^{N-(\gamma+1)p})$ from (\ref{ene}), then
$\mathcal{E}_{P}$ tends to $0$ at $\infty,$ which is contradictory. Next
consider the general case. We have
\begin{align*}
\mathcal{E}_{P}(r)  &  \leqq r^{N-(\gamma+1)p}\left[  X^{q(p-1)}%
Y^{p(q-1)}Z^{p(q-m-1)}W^{q(p-s-1)}\right]  ^{1/D}ZW\\
&  =r^{N-(\gamma+1)p}\left[  X^{q(p-1)}Y^{p(q-1)}Z^{q(1+s)}W^{p(1+m)}\right]
^{1/D},
\end{align*}
then the same result holds. Consequently, from Theorem \ref{equi}, there
exists a solution of the Dirichlet problem\medskip\ 

2) \textit{Behaviour of the }$G.S.$\textit{ in the critical case. }\medskip

Let $\mathcal{T}$ be the trajectory of a G.S.; then $\mathcal{E}_{P}(0)=0,$
thus $\mathcal{T}$ lies on the variety $\mathcal{V}$ of energy $0$, also
defined by
\begin{equation}
qW\left[  (s+1)((p-1)X-(N-p))+pZ\right]  =p(m+1)Z\left[  (N-q)-(q-1)Y)\right]
\label{war}%
\end{equation}
and $Y<\frac{N-q}{q-1},$ hence $(s+1)((p-1)X-(N-p))+pZ>0$. From (\ref{mil}),
$\mathcal{T}$ starts from $N_{0}=(0,0,N+a,N+b)$ and stays in $\mathcal{A}$.
Eliminating $W$ in system $(M),$ we find a system of three equations%
\[
\left\{
\begin{array}
[c]{c}%
X_{t}=X\left[  X-\frac{N-p}{p-1}+\frac{Z}{p-1}\right]  ,\\
Y_{t}=YF,\\
Z_{t}=Z\left[  N+a-sX-(m+1)Y-Z\right]  ,
\end{array}
\right.
\]
where
\[
F(X,Y,Z)=\frac{1}{q}\left[  \frac{N-q}{q-1}-Y\right]  \frac
{p(m+1-q)Z+q(s+1)((N-p)-(p-1)X)}{(s+1)((p-1)X-(N-p))+pZ}.
\]
(i) If $\mathcal{T}$ converges to a fixed point of the system in
$\mathcal{\bar{R}}$, the possible points on $\mathcal{V}$ are $A_{0}%
,I_{0},J_{0},$ $P_{0},$ $Q_{0},G_{0},H_{0},$ $R_{0},S_{0}.$ The eigenvalues of
the linearized problem at $A_{0},$ given by (\ref{fgh}) satisfy
\[
\lambda_{1},\lambda_{2}>0,\lambda_{3}=N+a-s\frac{N-p}{p-1}-(m+1)\frac
{N-q}{q-1}\leqq\lambda_{4}=\lambda^{\ast}=N+a-(s+1)\frac{N-p}{p-1}-m\frac
{N-q}{q-1},
\]
since $q\leqq p,$ and $\lambda_{3}<\lambda^{\ast}$ for $q\neq p,$ and
$\lambda_{3}=\lambda^{\ast}<0$ for $q=p,$ from (\ref{cho}). Then $A_{0}$ can
be attained only when $\lambda^{\ast}\leqq0,$ from Proposition \ref{cao}. And
$P_{0}$ can be attained only if
\begin{equation}
q>m+1,\;\lambda^{\ast}\geqq0\text{ }\;\text{and }\;q+a<(s+1)\frac{N-p}{p-1},
\label{aupe}%
\end{equation}
from Proposition \ref{cpo}, because $\gamma=\frac{N-p}{p}<\frac{N-p}{p-1}.$ We
observe that the condition $\lambda^{\ast}\geqq0$ joint to (\ref{cho}) implies
$m+1<q<p$ and is equivalent to (\ref{aupe}). Indeed it implies
\[
\frac{N-p}{p-1}(s+1)\leqq N+a-m\frac{N-q}{q-1}=N+a-\frac{q}{q-1}%
(N+a-\frac{N-q}{q}-\left(  s+1\right)  \frac{N-p}{p});
\]
then
\[
\left(  s+1\right)  \frac{N-p}{p-1}\frac{q-p}{p}\leqq-(a+q),
\]
thus $q<p.$ From (\ref{cho}) we obtain%
\[
(N-q)(\frac{m+1}{q}-1)=q+a-(s+1)\frac{N-p}{p}\leqq\left(  s+1\right)
\frac{N-p}{p}(\frac{p-q}{p-1}-1)<0,
\]
hence $m+1<q$ and (\ref{aupe}) follows. By symmetry, $Q_{0}$ cannot be
attained since $q\leqq p.$ Then $A_{0}$ and $P_{0}$ are incompatible, unless
$A_{0}=P_{0}$, and $P_{0}$ is not attained when $p=q.\medskip$

\noindent(ii) Next we show that $\mathcal{T}$ converges to $A_{0}$ or to
$P_{0}$. If $t$ is an extremum value of $Y$, then
\begin{equation}
(\frac{m+1}{q}-1)Z(t)+\frac{s+1}{p}((N-p)-(p-1)X(t))=0. \label{loi}%
\end{equation}
This relation implies $q>m+1$ and
\[
X_{t}(t)=\frac{X(t)Z(t)}{p-1}\left[  1+\frac{p(m+1-q)}{q(s+1)}\right]
=\frac{DX(t)Z(t)}{(p-1)q(s+1)}>0.
\]
In the same way, if $t$ is an extremum value of $X,$ then $p>s+1$ and
$Y_{t}(t)>0.$ Near $-\infty,$ there holds $X_{t},Y_{t}\geqq0,$ and
$Z_{t},W_{t}\leqq0,$ from the linearization near $N_{0}.$ Suppose that $X$ has
a maximum at $t_{0}$ followed by a minimum at $t_{1}.$ Then $p>s+1,$ and $Y$
is increasing on $\left[  t_{0},t_{1}\right]  $. At time $t_{0}$ we have
$(p-1)X(t_{0})+Z(t_{0})=N-p$ and $X_{tt}(t_{0})\leqq0,$ thus $Z_{t}%
(t_{0})\leqq0;$ eliminating $Z$ we deduce $p+a+(p-1-s)X(t_{0})\leqq
(m+1)Y(t_{0})$ and similarly $(m+1)Y(t_{1})\leqq p+a+(p-1-s)X(t_{1});$ hence
$Y(t_{1})<Y(t_{0})$, which is a contradiction. Thus $X$ and $Y$ can have at
most one maximum, and in turn they have no maximum point. Therefore $X$ and
$Y$ are increasing, and they are bounded, hence $X$ has a limit in $\left(
0,\frac{N-p}{p-1}\right]  $ and $Y$ has a limit in $\left(  0,\frac{N-q}%
{q-1}\right]  $. Then $Z,W$ are decreasing; indeed at each time where
$Z_{t}=0,$ we have $Z_{tt}=Z(-sX_{t}-(m+1)Y_{t})<0,$ thus it is a maximum,
which is impossible.\medskip

Then $\mathcal{T}$ converges to a fixed point of the system. Moreover, since
$X$ and $Y$ are increasing, it cannot be one of the points $I_{0},J_{0}%
,G_{0},H_{0},R_{0},S_{0}.$ It is necessarily $A_{0}$ or $P_{0}.$ We
distinguish two cases:\medskip\ 

$\bullet$ Case $q\leqq m+1$. Then $\mathcal{T}$ converges to $A_{0},$ and
$\lambda_{3},\lambda^{\ast}<0,$ then (\ref{iu}) follows.\medskip

$\bullet$ Case $q>m+1.$ Then $\mathcal{T}$ converges to $A_{0}$ (resp.
$P_{0})$ when $\lambda^{\ast}\leqq0$ (resp. $\lambda^{\ast}\geqq0)$. If the
inequalities are strict, we deduce the convergence of $u$ and $v$ from
Propositions \ref{cao} and \ref{cpo}, and (\ref{id}) follows. If
$\lambda^{\ast}=0,$ then $P_{0}=A_{0},$ and $\lambda_{3}=\frac{(N-1)(q-p)}%
{(p-1)(q-1)}<0.$ The projection of the trajectory $\mathcal{T}$ in
$\mathbb{R}^{3}$ on the plane $Y=\frac{N-q}{q-1}$ satisfies the system
\[
X_{t}=X\left[  X-\frac{N-p}{p-1}+\frac{Z}{p-1}\right]  ,\qquad Z_{t}=Z\left[
N+a-sX-(m+1)\frac{N-q}{q-1}-Z\right]
\]
which presents a saddle point at $(\frac{N-p}{p-1},0)$, thus the convergence
of $X$ and $Z$ is exponential, in particular we deduce the behaviour of $u.$
The trajectory enters by the central variety of dimension $1,$ and by
computation we deduce that $Y=\frac{N-q}{q-1}-\frac{1}{q-1-m}t^{-1}%
+O(t^{-2+\varepsilon}),$ then (\ref{it}) follows.
\end{proof}

\section{The nonradial potential system of Laplacians\label{nonradial}}

Here we study the possibly \underline{\textit{nonradial}} solutions of the
system of the preceeding Section when $p=q=2:$%
\[
(SL)\left\{
\begin{array}
[c]{c}%
-\Delta u=\left\vert x\right\vert ^{a}u^{s}v^{m+1},\\
-\Delta v=\left\vert x\right\vert ^{a}u^{s+1}v^{m},
\end{array}
\right.
\]
with $D=s+m.$ We solve an open problem of \cite{BR}: the nonexistence of
(radial or nonradial) G.S. under condition (\ref{naj}). \medskip

It was shown in \cite{BR} in the case $N+a\geqq4.$ The problem was open when
$N+a<4,$ and $m+s+1>(N+a)/(N-2),$ which implies $N<6.$ Indeed in the case
$m+s+1\leqq(N+a)/(N-2),$ there are no solutions of the exterior problem, see
\cite[Theorem 5.3]{BPo}. Recall that the main result of \cite{BR} is the
obtention of apriori estimates near $0$ or $\infty,$ by using the Bernstein
technique introduced in \cite{GiSp} and improved in \cite{BVe}. Then the
behaviour of the solutions is obtained by using the change of unknown
\[
u(r,\theta)=r^{-\gamma}\mathrm{U}(t,\theta),\qquad v(r,\theta)=r^{-\gamma
}\mathrm{V}(t,\theta),\qquad t=\ln r,
\]
extending the transformation of Remark \ref{bah} to the nonradial case (in
fact here $t$ is $-t$ in \cite{BR}$)$; it leads to the system
\begin{align*}
\mathrm{U}_{tt}+(N-2-2\gamma)\mathrm{U}_{t}+\Delta_{S}\mathrm{U}%
-\gamma(N-2-\gamma)\mathrm{U}+\mathrm{U}^{s}\mathrm{V}^{m+1}  &  =0,\\
\mathrm{V}_{tt}+(N-2-2\gamma)\mathrm{V}_{t}+\Delta_{S}\mathrm{V}%
-\gamma(N-2-\gamma)\mathrm{V}+\mathrm{U}^{s+1}\mathrm{V}^{m}  &  =0,
\end{align*}
where $\Delta_{S}$ is the Laplace-Beltrami operator on $S_{N-1}.$ A
corresponding energy is introduced in \cite{BR}:
\begin{align*}
E_{L}(t)  &  =\frac{s+1}{2}%
{\displaystyle\int\limits_{S^{N-1}}}
(\mathrm{U}_{t}^{2}-\left\vert \nabla_{S}\mathrm{U}\right\vert ^{2}%
-\gamma(N-2-\gamma)\mathrm{U}^{2})d\theta\\
&  +\frac{m+1}{2}%
{\displaystyle\int\limits_{S^{N-1}}}
(\mathrm{V}_{t}^{2}-\left\vert \nabla_{S}\mathrm{V}\right\vert ^{2}%
-\xi(N-2-\xi)\mathrm{V}^{2})d\theta+%
{\displaystyle\int\limits_{S^{N-1}}}
\mathrm{U}^{s+1}\mathrm{V}^{m+1}d\theta,
\end{align*}
extending (\ref{vvv}) to the nonradial case; it satisfies
\[
(E_{L})_{t}=-(N-2-2\gamma)%
{\displaystyle\int\limits_{S^{N-1}}}
\left[  (s+1)\mathrm{U}_{t}^{2}+(m+1)\mathrm{V}_{t}^{2}\right]  d\theta
\]
Here we construct \underline{\textit{another energy function}}, extending the
Pohozaev function defined at (\ref{ppo}) to the nonradial case.

\begin{lemma}
Consider the function $\mathcal{E}_{L}(r)$ defined by
\begin{align*}
r^{-N}\mathcal{E}_{L}(r)  &  =\frac{s+1}{2}%
{\displaystyle\int\limits_{S^{N-1}}}
\left[  u_{r}^{2}-r^{-2}\left\vert \nabla_{S}u\right\vert ^{2}+(N-2)\frac
{uu_{r}}{r}\right]  d\theta\\
&  +\frac{m+1}{2}%
{\displaystyle\int\limits_{S^{N-1}}}
\left[  (\frac{\partial v}{\partial\nu})^{2}-r^{-2}\left\vert \nabla
_{S}v\right\vert ^{2}+(N-2)\frac{vv_{r}}{r}\right]  d\theta+r^{a}%
{\displaystyle\int\limits_{S^{N-1}}}
u^{s+1}v^{m+1}d\theta.
\end{align*}
Then the following relation holds:
\[
r^{1-N}\mathcal{E}_{L}^{\prime}(r)=(N+a-(s+1)\frac{N-2}{2}-(m+1)\frac{N-2}%
{2})r^{a}%
{\displaystyle\int\limits_{S^{N-1}}}
u^{s+1}v^{m+1}d\theta.
\]

\end{lemma}

\begin{proof}
In terms of $t,$ we find%
\begin{align*}
\mathcal{E}_{L}(t)  &  =\mathcal{E}_{L,1}(t)+\mathcal{E}_{L,2}(t)+\mathcal{E}%
_{L,3}(t),\text{ with }\\
\mathcal{E}_{L,1}(t)  &  =\frac{s+1}{2}e^{(N-2)t}%
{\displaystyle\int\limits_{S^{N-1}}}
\left[  u_{t}^{2}-\left\vert \nabla_{S}u\right\vert ^{2}+(N-2)uu_{t}\right]
d\theta,\\
\mathcal{E}_{L,2}(t)  &  =\frac{m+1}{2}e^{(N-2)t}%
{\displaystyle\int\limits_{S^{N-1}}}
\left[  v_{t}^{2}-\left\vert \nabla_{S}v\right\vert ^{2}+(N-2)vv_{t}\right]
d\theta,\quad\mathcal{E}_{L,3}(t)=e^{(N+a)t}%
{\displaystyle\int\limits_{S^{N-1}}}
u^{s+1}v^{m+1}d\theta,
\end{align*}
and $u$ satisfies the equations
\begin{align}
u_{tt}+(N-2)u_{t}+\Delta_{S}u+e^{(2+a)t}u^{s}v^{m+1}  &  =0,\label{qua}\\
(e^{(N-2)t}u_{t})_{t}+e^{(N-2)t}\Delta_{S}u+e^{(N+a)t}u^{s}v^{m+1}  &  =0,
\label{quad}%
\end{align}
and $v$ satisfies symmetrical equations. Multiplying (\ref{quad}) by $u$ and
(\ref{qua}) by $(s+1)e^{(N-2)t}u_{t},$ we obtain
\begin{align*}
0  &  =%
{\displaystyle\int\limits_{S^{N-1}}}
u(e^{(N-2)t}u_{t})_{t}+e^{(N-2)t}%
{\displaystyle\int\limits_{S^{N-1}}}
u\Delta_{S}u+e^{(N+a)t}%
{\displaystyle\int\limits_{S^{N-1}}}
u^{s+1}v^{m+1}\\
&  =\frac{d}{dt}%
{\displaystyle\int\limits_{S^{N-1}}}
ue^{(N-2)t}u_{t}-e^{(N-2)t}%
{\displaystyle\int\limits_{S^{N-1}}}
(u_{t}^{2}+\left\vert \nabla_{S}u\right\vert ^{2})+e^{(N+a)t}%
{\displaystyle\int\limits_{S^{N-1}}}
u^{s+1}v^{m+1}%
\end{align*}%
\begin{align*}
&  \frac{d}{dt}%
{\displaystyle\int\limits_{S^{N-1}}}
\frac{s+1}{2}(N-2)ue^{(N-2)t}u_{t}-\frac{s+1}{2}(N-2)e^{(N-2)t}%
{\displaystyle\int\limits_{S^{N-1}}}
(u_{t}^{2}+\left\vert \nabla_{S}u\right\vert ^{2})\\
&  =-\frac{s+1}{2}(N-2)e^{(N+a)t}%
{\displaystyle\int\limits_{S^{N-1}}}
u^{s+1}v^{m+1},
\end{align*}
and symmetrically for $v,$ and adding the equalities we deduce%
\begin{align*}
0  &  =(s+1)(e^{(N-2)t}\frac{d}{dt}%
{\displaystyle\int\limits_{S^{N-1}}}
(\frac{u_{t}^{2}-\left\vert \nabla_{S}u\right\vert ^{2}}{2}+(N-2)e^{(N-2)t}%
{\displaystyle\int\limits_{S^{N-1}}}
u_{t}^{2}\\
&  +(m+1)(e^{(N-2)t}\frac{d}{dt}%
{\displaystyle\int\limits_{S^{N-1}}}
\frac{v_{t}^{2}-\left\vert \nabla_{S}v\right\vert ^{2}}{2}+(N-2)e^{(N-2)t}%
{\displaystyle\int\limits_{S^{N-1}}}
v_{t}^{2}\\
&  +\frac{d}{dt}(e^{(N+a)t}%
{\displaystyle\int\limits_{S^{N-1}}}
u^{s+1}v^{m+1})-(N+a)e^{(N+a)t}%
{\displaystyle\int\limits_{S^{N-1}}}
u^{s+1}v^{m+1}%
\end{align*}%
\begin{align*}
&  \frac{d}{dt}\left[  \frac{e^{(N-2)t}}{2}%
{\displaystyle\int\limits_{S^{N-1}}}
((s+1)(u_{t}^{2}-\left\vert \nabla_{S}u\right\vert ^{2})+(m+1)(v_{t}%
^{2}-\left\vert \nabla_{S}v\right\vert ^{2}))+e^{(N+a)t}%
{\displaystyle\int\limits_{S^{N-1}}}
u^{s+1}v^{m+1}\right] \\
&  +\frac{N-2}{2}e^{(N-2)t}%
{\displaystyle\int\limits_{S^{N-1}}}
((s+1)(u_{t}^{2}+\left\vert \nabla_{S}u\right\vert ^{2})+(m+1)(v_{t}%
^{2}+\left\vert \nabla_{S}v\right\vert ^{2}))\\
&  =(N+a)e^{(N+a)t}%
{\displaystyle\int\limits_{S^{N-1}}}
u^{s+1}v^{m+1},
\end{align*}
hence
\[
(\mathcal{E}_{L}\mathcal{)}_{t}(t)=(N+a-(s+1)\frac{N-2}{2}-(m+1)\frac{N-2}%
{2})e^{(N+a)t}%
{\displaystyle\int\limits_{S^{N-1}}}
u^{s+1}v^{m+1}d\theta.
\]
\medskip
\end{proof}

\begin{proof}
[Proof of Theorem \ref{solve}]Suppose that there exists a G.S. Since
$s+m+1<(N+2+2a)/(N-2)$ we deduce that $E_{L}$ and $\mathcal{E}_{L}$ are
increasing and start from 0, then they stay positive. From \cite[Corollary
6.4]{BR}, since $s+m+1<(N+2)/(N-2),$ three eventualities can hold. The first
one is that $(u,v)$ behaves like the particular solution $(u_{0},v_{0});$ it
cannot hold because $E_{L}$ has a negative limit, see \cite[Remark 6.3]{BR}.
The second one is that $(u,v)$ is regular at $\infty,$ that means
$\lim_{\left\vert x\right\vert \rightarrow\infty}\left\vert x\right\vert
^{N-2}u=\alpha>0,$ $\lim_{\left\vert x\right\vert \rightarrow\infty}\left\vert
x\right\vert ^{N-2}v=\beta>0;$ it cannot hold because $\lim_{t\rightarrow
\infty}E_{L}(t)=0.$ It remains a third eventuality: when for example
$m>(N+a)/(N-2)$, and $(u,v)$ has the following behaviour at $\infty:$
\begin{equation}
\lim_{r\rightarrow\infty}u=\alpha>0,\text{ and }\lim_{\left\vert x\right\vert
\rightarrow\infty}\left\vert x\right\vert ^{k}v=\beta>0\text{ or }0,\text{
\quad with }k=(2+a)/(m-1). \label{vict}%
\end{equation}
The condition on $m$ implies that $N<4-a$ from assumption (\ref{naj})$.$ In
that case $\lim_{t\rightarrow\infty}E_{L}(t)=\infty,$ which gives no
contradiction. Here we show that a contradiction holds by using the new energy
function $\mathcal{E}_{L}.\medskip$

First recall the proof of (\ref{vict}). Making the substitution
\[
u(r,\theta)=u(t,\theta),\qquad v(r,\theta)=r^{-k}\mathbf{V}(t,\theta),\qquad
t=\ln r,\theta\in S_{N-1},
\]
we get
\begin{equation}
\left\{
\begin{array}
[c]{c}%
u_{tt}+(N-2)u_{t}+\Delta_{S}u\mathbf{\;+}e^{-2kt}u^{s}\mathbf{V}^{m+1}=0,\\
\mathbf{V}_{tt}+(N-2-2k)\mathbf{V}_{t}+\Delta_{S}\mathbf{V}%
-k(N-2-k)\mathbf{V+}u^{s+1}\mathbf{V}^{m}=0.
\end{array}
\right.  \label{chos}%
\end{equation}
Then $u\mathbf{,V}$ are bounded near $\infty,$ and from \cite[Proposition
4.1]{BR} $u$ converges exponentially to the constant $\alpha,$ more precisely
\begin{equation}
\left\Vert \left\vert u\mathbf{-}\alpha\right\vert +\left\vert u_{t}%
\right\vert +\left\vert \nabla_{S}u\right\vert \right\Vert _{C^{0}(S^{N-1}%
)}=O(e^{-(N-2)t}), \label{mum}%
\end{equation}
because $k\neq(N-2)/2$ and all the derivatives of $\mathbf{V}$ up to the order
$2$ are bounded. The equation in $\mathbf{V}$ takes the form
\[
\mathbf{V}_{tt}+(N-2-2k)\mathbf{V}_{t}+\Delta_{S}\mathbf{V}%
-k(N-2-k)\mathbf{V+}\alpha^{s+1}\mathbf{V}^{m}+\varphi=0
\]
where $\varphi$ and its derivatives up to the order $2$ are $O(e^{-(N-2)t}).$
From \cite[Theorem 4.1]{BR}, the function $\mathbf{V}$ converges to $\beta$ or
to $0$ in $C^{2}(S^{N-1}).\medskip$

Next we define
\[
f(t)=e^{(N-2)t}%
{\displaystyle\int\limits_{S^{N-1}}}
u_{t}d\theta=r^{N-1}%
{\displaystyle\int\limits_{S^{N-1}}}
u_{r}d\theta.
\]
Then
\[
\mathcal{E}_{L,1}(t)=(N-2)\frac{s+1}{2}\alpha f(t)+O((e^{-(N-2)t}))
\]
from (\ref{mum}). Moreover from (\ref{chos}),
\[
f_{t}(t)=-e^{(N-2-2k)t}%
{\displaystyle\int\limits_{S^{N-1}}}
u^{s}\mathbf{V}^{m+1}d\theta<0.
\]
Since $u$ is regular at 0, $f(t)=0(e^{(N-1)t})$ at $-\infty,$ in particular
$\lim_{t\rightarrow-\infty}f(t)=0.$ And $f_{t}(t)=O(e^{(N-2-2k)t})=O(e^{-t})$
at $\infty,$ then $f(t)$ has a finite negative limit $-\ell^{2};$ and
\[
\lim_{t\rightarrow\infty}\mathcal{E}_{L,1}(t)=-(N-2)\frac{s+1}{2}\alpha
\ell^{2}.
\]
Moreover $v=e^{-kt}\mathbf{V},$ and $\mathbf{V}$ and its derivatives up to the
order $2$ are bounded, thus
\[
\mathcal{E}_{L,2}(t)=O(e^{(N-2-2k)t})=O(e^{-t})
\]
Finally
\[
\mathcal{E}_{L,3}(t)=O(e^{(N+a-k(m+1))t})
\]
and $N+a-k(m+1)<\frac{2-N}{m-1}<0.$ Then $\mathcal{E}_{L}$ has a finite limit
$\theta<0$ at $\infty,$ which is contradictory.
\end{proof}

\section{Analysis of the fixed points \label{app}}

Here we make the local analysis around the fixed points.\medskip

\begin{proof}
[Proof of Proposition \ref{ereg}](i) Consider a regular solution $(u,v)$ with
initial data $(u_{0},v_{0}).$ When when $r$ $\rightarrow$ $0,$ we have
\[
(-r^{N-1}\left\vert u^{\prime}\right\vert ^{p-2}u^{\prime})^{\prime}%
=r^{N-1+a}u_{0}^{s}v_{0}^{\delta}(1+o(1)),\qquad-\left\vert u^{\prime
}\right\vert ^{p-2}u^{\prime}=\frac{1}{N+a}r^{1+a}u_{0}^{s}v_{0}^{\delta
}(1+o(1)),
\]
thus from (\ref{xyt}), when $t\rightarrow-\infty$
\begin{align*}
X(t)  &  =(\frac{1}{N+a}u_{0}^{s+1-p}v_{0}^{\delta})^{1/(p-1)}e^{(p+a)t/(p-1)}%
(1+o(1)),\\
Y(t)  &  =(\frac{1}{N+b}u_{0}^{\mu}v_{0}^{m+1-q})^{1/(q-1)}e^{(q+b)t/(q-1)}%
(1+o(1)),
\end{align*}
and $\lim_{t\rightarrow-\infty}Z=N+a,$ $\lim_{t\rightarrow-\infty}W=(N+b).$ In
particular the trajectory tends to $N_{0}=(0,0,N+a,N+b).\medskip$

\noindent(ii) Reciprocally, consider a trajectory converging to $N_{0}.$
Setting $Z=N+a+\tilde{Z},W=N+b+\tilde{W},$ the linearized system is
\begin{equation}
X_{t}=\frac{p+a}{p-1}X,\quad Y_{t}=\frac{q+b}{q-1}Y,\quad\tilde{Z}%
_{t}=(N+a)\left[  -sX-\delta Y-\tilde{Z}\right]  ,\quad\tilde{W}%
_{t}=(N+b)\left[  -\mu X-mY-\tilde{W}\right]  . \label{lib}%
\end{equation}
The eigenvalues are
\begin{equation}
\lambda_{1}=\frac{p+a}{p-1}>0,\quad\lambda_{2}=\frac{q+b}{q-1}>0,\quad
\lambda_{3}=-(N+a)<0,\quad\lambda_{4}=-(N+b)<0. \label{vib}%
\end{equation}
The unstable variety $\mathcal{V}_{u}$ and the stable variety $\mathcal{V}%
_{s}$ have dimension $2.$ Notice that $\mathcal{V}_{s}$ is contained in the
set $X=Y=0,$ thus no admissible trajectory converges to $N_{0}$ when
$r\rightarrow\infty$, and there exists an infinity of admissible trajectories
in $\mathcal{R},$ converging to $N_{0}$ when $r\rightarrow0$. Moreover we get
$\lim_{t\rightarrow-\infty}e^{-(p+a)/(p-1)t}X(t)=\kappa>0$ and $\lim
_{t\rightarrow-\infty}e^{-(q+b)/(q-1)t}Y(t)=\ell>0$. Thus $(u,v)$ have a
positive limit $(u_{0},v_{0})=\mathbf{(}(N+a)\kappa^{p-1})^{(q-1-m)/D}%
\mathbf{(}(N+b)\ell^{q-1})^{\delta/D}$ from (\ref{for}), (\ref{eto}), hence
$(u,v)$ is a regular solution.$\medskip$

Next we show that \textit{for any }$\kappa>0,\ell>0$ there exists a unique
local solution such that $\lim_{t\rightarrow-\infty}e^{-(p+a)t/(p-1)}%
X(t)=\kappa$ and $\lim_{t\rightarrow-\infty}e^{-(q+b)/(q-1)t}Y=\ell.$ On
$\mathcal{V}_{u}$, we get a system of two equations of the form
\[
X_{t}=X(\lambda_{1}+F(X,Y)),\quad Y_{t}=Y(\lambda_{2}+G(X,Y)),
\]
where $F=AX+BY+f(X,Y),$ where $f$ is a smooth function with $f_{X}%
(0,0)=f_{Y}(0,0)=0,$ similarly for $G.$ Setting $X=e^{\lambda_{1}t}%
(\kappa+x),$ $Y=e^{\lambda_{2}t}(\ell+y),$ and assuming $\lambda_{2}%
\geqq\lambda_{1}$ and setting $\rho=e^{\lambda_{1}t}$ we obtain
\[
x_{\rho}=\frac{1}{\rho}(\kappa+x)F(\rho(\kappa+x),\rho^{\lambda_{2}%
/\lambda_{1}}(\ell+y)),\qquad y_{\rho}=(\ell+y)G(\rho(\kappa+x),\rho
^{\lambda_{2}/\lambda_{1}}(\ell+y)),
\]
with $x(0)=y(0)=0.$ Then we get local existence and uniqueness. Hence for any
$u_{0},v_{0}>0$ there exists a regular solution $(u,v)$ with initial data
($u_{0},v_{0}).$ Moreover $u,v\in$ $C^{1}(\left[  0,R\right)  )$ when
$a,b>-1.\medskip$
\end{proof}

\begin{proof}
[Proof of Proposition \ref{cao}]The linearization at $A_{0}=\left(  \frac
{N-p}{p-1},\frac{N-q}{q-1},0,0\right)  $ gives, with $X=\frac{N-p}{p-1}%
+\tilde{X},Y=\frac{N-q}{q-1}+\tilde{Y},$%
\[
\tilde{X}_{t}=\frac{N-p}{p-1}\left[  \tilde{X}+\frac{Z}{p-1}\right]
,\quad\tilde{Y}_{t}=\frac{N-q}{q-1}\left[  \tilde{Y}+\frac{W}{q-1}\right]
,\quad Z_{t}=\lambda_{3}Z,\quad W_{t}=\lambda_{4}W.
\]
The eigenvalues are
\begin{equation}
\lambda_{1}=\frac{N-p}{p-1}>0,\;\lambda_{2}=\frac{N-q}{q-1}>0,\;\lambda
_{3}=N+a-s\frac{N-p}{p-1}-\delta\frac{N-q}{q-1},\;\lambda_{4}=N+b-\mu
\frac{N-p}{p-1}-m\frac{N-q}{q-1}. \label{fgh}%
\end{equation}
$\bullet$ Convergence when $r\rightarrow$ $\infty:$ If $\lambda_{3}>0,$ or
$\lambda_{4}>0,$ then the stable variety$\mathcal{V}_{s}$ has at most
dimension $1,$ it satisfies $W=0$ or $Z=0,$ hence there is no admissible
trajectory converging to $A_{0}$ at $\infty.$ If $\lambda_{3}<0,$ and
$\lambda_{4}<0,$ then $\mathcal{V}_{s}$ has dimension $2.$ Moreover
$\mathcal{V}_{s}\cap\left\{  Z=0\right\}  $ has dimension $1:$ the
corresponding system in $X,Y,W$ has the eigenvalues $\lambda_{1},\lambda
_{2},\lambda_{4};$ similarly $\mathcal{V}_{s}\cap\left\{  W=0\right\}  $ has
dimension 1. Then there exist trajectories in $\mathcal{V}_{s}$ such that
$Z>0$ and $W>0,$ included in $\mathcal{R}$ and thus admissible. They satisfy
$\lim e^{-\lambda_{3}t}Z=C_{3}>0,\lim e^{-\lambda_{4}t}W=C_{4}>0,$ then
(\ref{oup}) follows from (\ref{for}).\medskip

\noindent$\bullet$ Convergence when $r\rightarrow0:$ If $\lambda_{3}<0,$ or
$\lambda_{4}<0,$ the unstable variety $\mathcal{V}_{u}$ has at most dimension
$3$, and it satisfies $W=0$ or $Z=0.$ Therefore there is no admissible
trajectory converging at $-\infty.$ If $\lambda_{3},\lambda_{4}>0,$ then
$\mathcal{V}_{u}$ has dimension $4;$ in that case there exist admissible
trajectories, and (\ref{oup}) follows as above.\medskip
\end{proof}

\begin{proof}
[Proof of Proposition \ref{cpo}]We set $P_{0}=\left(  \frac{N-p}{p-1},Y_{\ast
},0,W_{\ast}\right)  ,$ with
\[
Y_{\ast}=\frac{\frac{N-p}{p-1}\mu-(q+b)}{q-1-m},\qquad W_{\ast}=\frac
{(q-1)(N+b-\frac{N-p}{p-1}\mu)-m(N-q)}{q-1-m},
\]
for $m+1\neq q.$ The linearization at $P_{0}$ gives, with $X=\frac{N-p}%
{p-1}+\tilde{X},Y=Y_{\ast}+\tilde{Y},$ $W=W_{\ast}+\tilde{W},$
\[
\tilde{X}_{t}=\frac{N-p}{p-1}\left[  \tilde{X}+\frac{Z}{p-1}\right]
,\quad\tilde{Y}_{t}=Y_{\ast}\left[  \tilde{Y}+\frac{\tilde{W}}{q-1}\right]
,\quad Z_{t}=\lambda_{3}Z,\quad\tilde{W}_{t}=W_{\ast}\left[  -\mu\tilde
{X}-m\tilde{Y}-\tilde{W}\right]
\]
The eigenvalues are
\[
\lambda_{1}=\frac{N-p}{p-1}>0,\quad\lambda_{3}=N+a-s\frac{N-p}{p-1}-\delta
Y_{\ast}=\frac{D}{q-1-m}(\gamma-\frac{N-p}{p-1}),
\]
and the roots $\lambda_{2},\lambda_{4}$ of equation
\[
\lambda^{2}-(Y_{\ast}-W_{\ast})\lambda+\frac{m+1-q}{q-1}Y_{\ast}W_{\ast}=0
\]
Then if $\lambda_{3}<0$ (resp. $\lambda_{3}>0)$ there is no admissible
trajectory converging when $r\rightarrow0$ (resp. $r\rightarrow\infty).$
Indeed $\mathcal{V}_{u}=\mathcal{V}_{u}\cap\left\{  Z=0\right\}  $ (resp.
$\mathcal{V}_{s}=\mathcal{V}_{s}\cap\left\{  Z=0\right\}  )$.\medskip

1) Suppose that $q>m+1.$ Since $q+b<\frac{N-p}{p-1}\mu<N+b-m\frac{N-q}{q-1},$
we have $P_{0}\in\mathcal{R},$ and $\lambda_{2}\lambda_{4}<0$. First assume
$\lambda_{3}<0,$ that means $\gamma<\frac{N-p}{p-1}$. Then $\mathcal{V}_{s}$
has dimension 2, and $\mathcal{V}_{s}\cap\left\{  Z=0\right\}  $ has dimension
1, there exists trajectories with $Z>0,$ which are admissible, converging when
$r\rightarrow\infty.$ Next assume $\lambda_{3}>0$. Then $\mathcal{V}_{u}$ has
dimension 3, and $\mathcal{V}_{u}\cap\left\{  Z=0\right\}  $ has dimension 2.
Thus there exist admissible trajectories converging when $t\rightarrow
-\infty.$\medskip

2) Suppose that $q<m+1$. Since $q+b>\frac{N-p}{p-1}\mu>N+b-m\frac{N-q}{q-1},$
we have $P_{0}\in\mathcal{R},$ and $\lambda_{2}\lambda_{4}>0$. We assume
$q\frac{N-p}{p-1}\mu+m(N-q)\neq N(q-1)+(b+1)q,$ that means $Y_{\ast}\neq
W_{\ast}.$ First suppose $\lambda_{3}>0,$ that means $\gamma<\frac{N-p}{p-1}.$
If Re$\lambda_{2}>0,$ then $\mathcal{V}_{u}$ has dimension 4, or
Re$\lambda_{2}<0$ then $\mathcal{V}_{u}$ has dimension $2$ and $\mathcal{V}%
_{u}\cap\left\{  Z=0\right\}  $ has dimension $1.$ In any case, there exist
admissible trajectories converging when $r$ $\rightarrow$ $0$. Next assume
$\lambda_{3}<0$. If Re$\lambda_{2}>0,$ then $\mathcal{V}_{s}$ has dimension 1,
and $\mathcal{V}_{s}\cap\left\{  Z=0\right\}  =\emptyset.$ If Re$\lambda
_{2}<0,$ then $\mathcal{V}_{s}$ has dimension $3$. In any case $\mathcal{V}%
_{s}$ contains trajectories with $Z>0,$ which are admissible, converging when
$r\rightarrow\infty.$

Those trajectories satisfy $\lim e^{-\lambda_{3}t}Z=C_{3}>0,$ $\lim
X=\frac{N-p}{p-1},$ $\lim Y=Y_{\ast}$ and $\lim W=W_{\ast},$ thus (\ref{aup})
follows from (\ref{for}) and (\ref{gk}).\medskip
\end{proof}

\begin{proof}
[Proof of Proposition \ref{mis}]The linearization at $I_{0}=(\frac{N-p}%
{p-1},0,0,0)$ gives, with $X=\frac{N-p}{p-1}+\tilde{X},$
\[
\tilde{X}_{t}=\frac{N-p}{p-1}(\tilde{X}+\frac{Z}{p-1}),\quad Y_{t}=-\frac
{N-q}{q-1}Y,\quad Z_{t}=(N+a-s\frac{N-p}{p-1})Z,\quad W_{t}=(N+b-\mu\frac
{N-p}{p-1})W.
\]
The eigenvalues are
\[
\lambda_{1}=\frac{N-p}{p-1}>0,\quad\lambda_{2}=-\frac{N-q}{q-1}<0,\quad
\lambda_{3}=N+a-s\frac{N-p}{p-1},\quad\lambda_{4}=N+b-\mu\frac{N-p}{p-1}.
\]
$\bullet$ Convergence when $r\rightarrow$ $\infty:$ If $\lambda_{3}>0$ or
$\lambda_{4}>0,$ then $\mathcal{V}_{s}=\mathcal{V}_{s}\cap\left\{
Z=0\right\}  $ or $\mathcal{V}_{s}=\mathcal{V}_{s}\cap\left\{  W=0\right\}  $.
There is no admissible trajectory converging at $\infty$. Next suppose that
$\lambda_{3},\lambda_{4}<0.$ Then $\mathcal{V}_{s}$ has dimension $3;$ it
contains trajectories with $Y,Z,W>0,$ which are admissible. They satisfy $\lim
X=\frac{N-p}{p-1},$ $\lim e^{-\lambda_{2}t}Y=C_{2}>0,$ $\lim e^{-\lambda_{3}%
t}Z=C_{3}>0,$ $\lim e^{-\lambda_{4}t}W=C_{4}>0,$ then (\ref{lup}) follows from
(\ref{for}) and (\ref{ga}).\medskip

\noindent$\bullet$ Convergence when $r\rightarrow$ $0:$ Since $\lambda_{2}<0$
we have $\mathcal{V}_{u}=\mathcal{V}_{u}\cap\left\{  Y=0\right\}  ,$ hence
there is no admissible trajectory converging when $r\rightarrow0.$ \medskip
\end{proof}

\begin{proof}
[Proof of Proposition \ref{gzero}]The point $G_{0}=(\frac{N-p}{p-1}%
,0,0,N+b-\frac{N-p}{p-1}\mu)\in\mathcal{R}$ since $\frac{N-p}{p-1}\mu<N+b.$
The linearization at $G_{0}$ gives, with $X=\frac{N-p}{p-1}+\tilde
{X},W=N+b-\frac{N-p}{p-1}\mu+\tilde{W},$
\begin{align*}
\tilde{X}_{t}  &  =\frac{N-p}{p-1}\left[  \tilde{X}+\frac{Z}{p-1}\right]
,\quad Y_{t}=\frac{Y}{q-1}(q+b-\frac{N-p}{p-1}\mu),\\
Z_{t}  &  =(N+a-s\frac{N-p}{p-1})Z,\quad W_{t}=(N+b-\frac{N-p}{p-1}\mu)\left[
-\mu\tilde{X}-mY-\tilde{W}\right]
\end{align*}
The eigenvalues are
\[
\lambda_{1}=\frac{N-p}{p-1}>0,\;\lambda_{2}=\frac{1}{q-1}(q+b-\frac{N-p}%
{p-1}\mu),\;\lambda_{3}=N+a-s\frac{N-p}{p-1},\;\lambda_{4}=\frac{N-p}{p-1}%
\mu-N-b<0.
\]

\noindent$\bullet$ Convergence when $r\rightarrow$ $\infty:$ If $\lambda
_{2}>0,$ or $\lambda_{3}>0,$ then $\mathcal{V}_{s}=$ $\mathcal{V}_{s}%
\cap\left\{  Y=0\right\}  $ or $\mathcal{V}_{s}=\mathcal{V}_{s}\cap\left\{
Z=0\right\}  $, there is no admissible trajectory converging at $\infty$.
Assume $\lambda_{2},\lambda_{3}<0,$ then $\mathcal{V}_{s}$ has dimension 3, it
contains trajectories with $Y,Z>0,$ which are admissible.\medskip

\noindent$\bullet$ Convergence when $r\rightarrow$ $0:$ If $\lambda_{3}<0,$ or
$\lambda_{2}<0$ there is no admissible trajectory. If $\lambda_{2},\lambda
_{3}>0$ then $\mathcal{V}_{s}$ has dimension $3,$ it contains admissible
trajectories.\medskip

In any case $\lim X=\frac{N-p}{p-1},$ $\lim e^{-\lambda_{2}t}Y=C_{2}>0,$ $\lim
e^{-\lambda_{3}t}Z=C_{3}>0,$ $\lim W=N+b-\frac{N-p}{p-1}\mu,$ hence
(\ref{lup}) still follows from (\ref{for}) and (\ref{ga}).\medskip
\end{proof}

\begin{proof}
[Proof of Proposition \ref{mes}]We set $C_{0}=\left(  0,\bar{Y},0,\bar
{W}\right)  ,$ with
\begin{equation}
\bar{Y}=\frac{q+b}{m+1-q},\qquad\bar{W}=\frac{m(N-q)-(N+b)(q-1)}{m+1-q}.
\label{bcd}%
\end{equation}
Then $C_{0}\in\mathcal{R}$ if $\frac{N-q}{q-1}m>N+b,$ implying $q<m+1.$ The
linearization at $C_{0}$ gives, with $Y=\bar{Y}+\tilde{Y}$ and $W=\bar
{W}+\tilde{W}$
\begin{equation}
X_{t}=-\frac{N-p}{p-1}X,\quad\tilde{Y}_{t}=\bar{Y}\left[  \tilde{Y}%
+\frac{\tilde{W}}{q-1}\right]  ,\quad Z_{t}=\lambda_{3}Z,\quad W_{t}=\bar
{W}\left[  -\mu X-m\tilde{Y}-\tilde{W}\right]  .\nonumber
\end{equation}
The eigenvalues are
\[
\lambda_{1}=-\frac{N-p}{p-1},\quad\lambda_{3}=N+a-\delta\bar{Y},
\]
and the roots $\lambda_{2},\lambda_{4}$ of equation%
\begin{equation}
\lambda^{2}-(\bar{Y}-\bar{W})\lambda+\frac{m+1-q}{q-1}\bar{Y}\bar{W}=0
\label{lad}%
\end{equation}
then $\lambda_{2}\lambda_{4}>0.$ We assume $m\neq\frac{N(q-1)+(b+1)q}{N-q},$
that means $\bar{Y}\neq\bar{W}.$\medskip

\noindent$\bullet$ Convergence when $r\rightarrow$ $\infty:$ if $\lambda
_{3}>0$ we have $\mathcal{V}_{s}=\mathcal{V}_{s}\cap\left\{  Z=0\right\}  ,$
hence there is no admissible trajectory$.$ Next assume that $\lambda_{3}<0$,
that means $\delta>(N+a)\frac{m+1-q}{q+b}.$If $\operatorname{Re}\lambda_{2}<0$
(resp. $>0)$ then $\mathcal{V}_{s}$ has dimension $4$ (resp. $2$) and
$\mathcal{V}_{s}\cap\left\{  X=0\right\}  $ and $\mathcal{V}_{s}\cap\left\{
Z=0\right\}  $ have dimension $3$ (resp. $1$) then there exist trajectories
with $X,Z>0,$ which are admissible.\medskip

In any case $\lim e^{-\lambda_{1}t}X=C_{1}>0,$ $\lim Y=\bar{Y},$ $\lim
e^{-\lambda_{3}t}Z=C_{3}>0,$ $\lim W=\bar{W},$ then (\ref{rup})
follows.\medskip

\noindent$\bullet$ Convergence when $r\rightarrow$ $0:$ Since $\lambda_{1}<0$
we have $\mathcal{V}_{u}=\mathcal{V}_{u}\cap\left\{  X=0\right\}  ,$ hence
there is no admissible trajectory.\medskip
\end{proof}

\begin{proof}
[Proof of Proposition \ref{mas}]We set $R_{0}=\left(  0,\bar{Y},\bar{Z}%
,\bar{W}\right)  ,$ where $\bar{Y},\bar{W}$ are defined at (\ref{bcd}), and
$\bar{Z}=N+a-\delta\frac{b+q}{m+1-q}.$Under our assumptions it lies in
$\mathcal{R}$. Setting $Y=\bar{Y}+\tilde{Y},Z=\bar{Z}+\tilde{Z},W=\bar
{W}+\tilde{W},$ the linearization at $R_{0}$ gives%
\begin{equation}
X_{t}=\lambda_{1}X,\quad\tilde{Y}_{t}=\bar{Y}\left[  \tilde{Y}+\frac{\tilde
{W}}{q-1}\right]  ,\quad Z_{t}=\bar{Z}\left[  -sX-\delta\tilde{Y}-\tilde
{Z}\right]  ,\quad W_{t}=\bar{W}\left[  -\mu X-m\tilde{Y}-\tilde{W}\right]
;\nonumber
\end{equation}
the eigenvalues are
\[
\lambda_{1}=\frac{1}{p-1}(p+a-\delta\frac{b+q}{m+1-q}),\quad\lambda_{3}%
=-\bar{Z}<0;
\]
and the roots $\lambda_{2},\lambda_{4}$ of equation of equation (\ref{lad}%
).\medskip

\noindent$\bullet$ Convergence when $r\rightarrow$ $\infty:$ If $\lambda
_{1}>0$, that means $(p+a)\frac{m+1-q}{q+b}<\delta,$ then $\mathcal{V}%
_{s}=\mathcal{V}_{s}\cap\left\{  X=0\right\}  ,$ hence there is no admissible
trajectory. Next assume $\lambda_{1}<0;$ if $\operatorname{Re}\lambda_{2}<0$
(resp. $>0)$ then $\mathcal{V}_{s}$ has dimension 4(resp. $2$) and
$\mathcal{V}_{s}\cap\left\{  X=0\right\}  $ has dimension $3$ (resp. $1$) then
there exist admissible trajectories.\medskip

\noindent$\bullet$ Convergence when $r\rightarrow$ $0:$ If $\lambda_{1}<0$,
then $\mathcal{V}_{u}=\mathcal{V}_{u}\cap\left\{  X=0\right\}  ,$ hence there
is no admissible trajectory. Next assume $\lambda_{1}>0.$ If
$\operatorname{Re}\lambda_{2}=\operatorname{Re}\lambda_{4}<0$ (resp. $>0)$
then $\mathcal{V}_{s}$ has dimension 4 (resp. $2$) and $\mathcal{V}_{s}%
\cap\left\{  X=0\right\}  $ has dimension $3$ (resp. $1$) then there exist
admissible trajectories.\medskip

In any case $\lim e^{-\lambda_{1}t}X=C_{1}>0,$ $\lim Y=\bar{Y},$ $\lim
Z=\bar{Z},$ $\lim W=\bar{W},$ then (\ref{rup}) holds again.\medskip
\end{proof}

\begin{remark}
\label{none}Finally there is no admissible trajectory converging to
$0=(0,0,0,0),$ or $K_{0}=(0,0,N+a,0),$ or $L_{0}=(0,0,0,N+b)$. Indeed the
linearization at $0$ gives
\[
X_{t}=-\frac{N-p}{p-1}X,\quad Y_{t}=-\frac{N-q}{q-1}Y,\quad Z_{t}=(N+a)Z,\quad
W_{t}=(N+b)W
\]
Then $\mathcal{V}_{s}$ and $\mathcal{V}_{u}$ have dimension $2,$ hence
$\mathcal{V}_{s}$ is contained in $\left\{  Z=W=0\right\}  ,$ and
$\mathcal{V}_{u}$ in $\left\{  X=Y=0\right\}  .$ The linearization at $K_{0}$
gives, with $Z=N+a+\tilde{Z},$
\[
X_{t}=\frac{p+a}{p-1}X,\quad Y_{t}=-\frac{N-q}{q-1}Y,\quad Z_{t}=(N+a)\left[
-sX-\delta Y-\tilde{Z}\right]  ,\quad W_{t}=(N+b)W.
\]
The eigenvalues are $\frac{p+a}{p-1},-\frac{N-q}{q-1},-(N+a),$ $N+b.$ Then
$\mathcal{V}_{s}$ and $\mathcal{V}_{u}$ have dimension $2,$ hence
$\mathcal{V}_{s}$ is contained in $\left\{  Z=W=0\right\}  $, and
$\mathcal{V}_{u}$ in $\left\{  Y=0\right\}  .$ The case of $L_{0}$ follows by symmetry.
\end{remark}

\end{document}